\documentclass[12pt]{amsart}
\usepackage{amsmath,amssymb,amsfonts,amscd,amsthm}
\usepackage{tikz,pgflibraryarrows}
\usepackage{hyperref}

\newtheorem{theorem}{Theorem}[section]
\newtheorem{lemma}[theorem]{Lemma}
\newtheorem{proposition}[theorem]{Proposition}

\newtheorem{corollary}[theorem]{Corollary}

\newtheorem*{theorem*}{Theorem}
\theoremstyle{definition}
\newtheorem{remark}[theorem]{Remark}
\newtheorem{definition}[theorem]{Definition}
\newtheorem{example}[theorem]{Example}

\numberwithin{equation}{section}

\newcommand{\N}{\mathbb{N}}
\newcommand{\R}{\mathbb{R}}
\newcommand{\C}{\mathbb{C}}

\newcommand{\Hi}{\mathcal{H}}
\newcommand{\Bi}{\mathcal{B}}
\newcommand{\K}{\mathcal{K}}

\newcommand{\G}{\mathcal{G}}

\newcommand{\cC}{\mathcal{C}}

\newcommand{\wt}[1]{\widetilde{#1}}
\newcommand{\Pow}{\mathcal{P}}

 \DeclareMathOperator{\cspa}{\overline{span}}

\newcommand{\Ca}{$C^*$-al\-ge\-bra }

\newcommand{\shom}{$*$-ho\-mo\-mor\-phism }

\newcommand{\sinks}{\operatorname{sinks}}
\newcommand{\MU}[1]{
\setbox0\hbox{$#1$}
{#1}^+ }

\begin{document}
\title[Realizations of AF-algebras]%
{Realizations of AF-algebras as graph algebras, Exel-Laca
algebras, and ultragraph algebras}

\author{Takeshi Katsura}
\address{Takeshi Katsura, Department of Mathematics\\ Keio University\\
Yokohama, 223-8522\\ JAPAN}
\email{katsura@math.keio.ac.jp}

\author{Aidan Sims}
\address{Aidan Sims, School of Mathematics and Applied Statistics\\
University of Wollongong\\
NSW 2522\\
AUSTRALIA} \email{asims@uow.edu.au}

\author{Mark Tomforde}
\address{Mark Tomforde \\ Department of Mathematics\\ University of Houston\\
Houston \\ TX 77204-3008\\ USA} \email{tomforde@math.uh.edu}

\date{October 22, 2008; revised April 29, 2009}


\begin{abstract}
We give various necessary and sufficient conditions for an
AF-algebra to be isomorphic to a graph $C^*$-algebra, an
Exel-Laca algebra, and an ultragraph $C^*$-algebra.  We also explore
consequences of these results.  In particular, we show that
all stable AF-algebras are both graph $C^*$-algebras and
Exel-Laca algebras, and that all simple AF-algebras are either
graph $C^*$-algebras or Exel-Laca algebras.  In addition, we
obtain a characterization of AF-algebras that are isomorphic to
the $C^*$-algebra of a row-finite graph with no sinks.
\end{abstract}

\maketitle

\tableofcontents

\section{Introduction} \label{intro-sec}
In 1980 Cuntz and Krieger introduced a class of $C^*$-algebras
constructed from finite matrices with entries in $\{0, 1 \}$
\cite{CK}.  These $C^*$-algebras, now called Cuntz-Krieger
algebras, are intimately related to the dynamics of topological
Markov chains, and appear frequently in many diverse areas of
$C^*$-algebra theory.  Cuntz-Krieger algebras have been
generalized in a number of ways, and two very natural
generalizations are the \emph{graph $C^*$-algebras} and the
\emph{Exel-Laca algebras}.

For graph $C^*$-algebras one views a $\{0, 1 \}$-matrix as an
edge adjacency matrix of a graph, and considers the
Cuntz-Krieger algebras as $C^*$-algebras of certain finite
directed graphs.  For a (not necessarily finite) directed graph
$E$, one then defines the graph $C^*$-algebra $C^*(E)$ as the
$C^*$-algebra generated by projections $p_v$ associated to the
vertices $v$ of $E$ and partial isometries $s_e$ associated to
the edges $e$ of $E$ that satisfy relations determined by the
graph.  Graph $C^*$-algebras were first studied using groupoid
methods \cite{KPR, KPRR}. Due to technical constraints, the
original theory was restricted to graphs that are
\emph{row-finite} and have \emph{no sinks}; that is, the set of
edges emitted by each vertex is finite and nonempty.  In fact
much of the early theory restricted to this case \cite{BPRS,
KPR, KPRR}, and it was not until later \cite{BHRS, DT, FLR}
that the theory was extended to infinite graphs that are not
row-finite. Interestingly, the non-row-finite setting is
significantly more complicated than the row-finite case, with
both new isomorphism classes of $C^*$-algebras and new kinds of
$C^*$-algebraic phenomena exhibited.

Another approach to generalizing the Cuntz-Krieger algebras was
taken by Exel and Laca, who defined what are now called the
Exel-Laca algebras \cite{EL}.  In this definition one allows a
possibly infinite matrix with entries in $\{0, 1 \}$ and
considers the $C^*$-algebra generated by a set of
partial isometries indexed by the rows of the matrix and
satisfying certain relations determined by the matrix.  The
construction of the Exel-Laca algebras contains the
Cuntz-Krieger construction as a special case.  Furthermore, for
\emph{row-finite matrices} (i.e., matrices in which each row
contains a finite number of nonzero entries) with nonzero rows,
the construction produces exactly the class of $C^*$-algebras
of row-finite graphs with no sinks.

Despite the fact that the classes of graph $C^*$-algebras and
Exel-Laca algebras agree in the row-finite case, they are quite
different in the non-row-finite setting.  In particular, there
are $C^*$-algebras of non-row-finite graphs that are not
isomorphic to any Exel-Laca algebra, and there are Exel-Laca
algebras of non-row-finite matrices that are not isomorphic to
the $C^*$-algebra of any graph \cite{Tom2}.  In order to bring
graph $C^*$-algebras and Exel-Laca algebras together under one
theory, Tomforde introduced the notion of an ultragraph and
described how to associate a $C^*$-algebra to such an object
\cite{Tom, Tom2}. These ultragraph $C^*$-algebras contain all
graph $C^*$-algebras and all Exel-Laca algebras, as well as
examples of $C^*$-algebras that are in neither of these two
classes.  The relationship among these classes is summarized in
Figure~\ref{fig:three-classes}.

\begin{figure}[htp!]
\[\begin{tikzpicture}[>=latex]]
  \draw[style=thick,fill=white,opacity=1] (1.5,0) circle (3.5);
  \node[anchor=north east] at (4.6,2) {\begin{tabular}{c}\textsc{Exel-Laca}\\ \textsc{Algebras}\end{tabular}};
  \draw[style=thick] (-1.5,0) circle (3.5);
  \node[anchor=north west] at (-4.95,2) {\begin{tabular}{c}\textsc{Graph}\\ \textsc{$C^*$-algebras}\end{tabular}};
  \draw[style=thick, smooth cycle] plot[tension=0.7] coordinates{(-4.5,-3.0) (-4.5,3.5) (4.5,3.5) (4.5,-3.0)};
  \node at (0,3.85) {\textsc{Ultragraph $C^*$-algebras}};
  \draw[style=semithick, smooth cycle] plot[tension=0.7] coordinates{(-1,-2) (-1,2) (1,2) (1,-2)};
          \node[anchor=north west, inner sep=3pt] at (-1.35,1.4) {\begin{tabular}{c} \text{$C^*$-algebras} \\ \text{of row-finite} \\ \text{graphs}  \\ \text{with no} \\ \text{sinks} \end{tabular}};
\end{tikzpicture}\]
\caption{The relationship among graph $C^*$-algebras, Exel-Laca algebras, and ultragraph $C^*$-algebras}\label{fig:three-classes}
\end{figure}
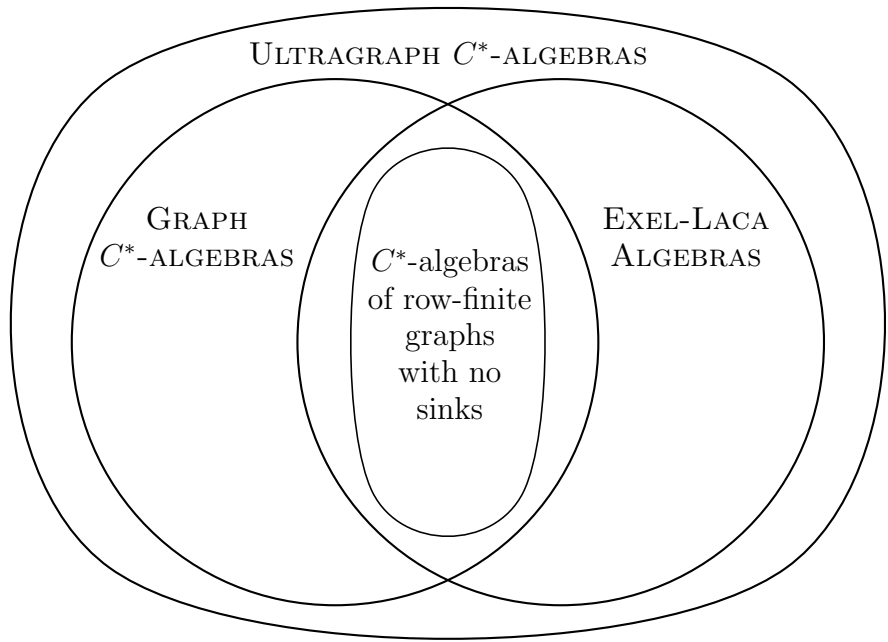

Given the relationship among these classes of $C^*$-algebras,
it is natural to ask the following question.

\vskip1ex

\noindent \textbf{Question: }``How different are the
$C^*$-algebras in the three classes of graph $C^*$-algebras,
Exel-Laca algebras, and ultragraph $C^*$-algebras?"

\vskip1ex

\noindent There are various ways to approach this question, and
one such approach was taken in \cite{KMST2}, where it was shown
that the classes of graph $C^*$-algebras, Exel-Laca algebras,
and ultragraph $C^*$-algebras agree up to Morita equivalence.
More specifically, given a $C^*$-algebra $A$ in any of these
three classes, one can always find a row-finite graph $E$ with
no sinks such that $C^*(E)$ is Morita equivalent to $A$.

Thus the three classes cannot be distinguished by Morita
equivalence classes of $C^*$-algebras. The natural next
question is to what extent they can be distinguished by isomorphism
classes of $C^*$-algebras. A starting point for these
investigations is to ask about AF-algebras.

While no Cuntz-Krieger algebra is an AF-algebra, the classes of graph $C^*$-algebras and Exel-Laca algebras each include many AF-algebras.  In fact, one of the early results
in the theory of graph $C^*$-algebras shows that if $A$ is any
AF-algebra, then there is a row-finite graph $E$ with no sinks
such that $C^*(E)$ is Morita equivalent to $A$
\cite{Drinen2000}.
From this fact and the result in \cite{KMST2} mentioned above,
our three classes (graph $C^*$-algebras, Exel-Laca algebras, and ultragraph
$C^*$-algebras) each contain all AF-algebras up to Morita
equivalence.

The purpose of this paper is to examine the three classes of
graph $C^*$-algebras, Exel-Laca algebras, and ultragraph
$C^*$-algebras and determine which AF-algebras are contained,
up to isomorphism, in each class.  This turns out to be a
difficult task, and we are unable to give a complete solution
to the problem.  Nonetheless, we are able to give a number of
sufficient conditions and a number of necessary conditions for
a given AF-algebra to belong to each of these three classes
(see
\S\ref{sufficient-cond-subsec}~and~\S\ref{necessary-cond-subsec}).
As special cases of our sufficient conditions, we obtain the
following.
\begin{itemize}
\item If $A$ is a stable AF-algebra, then $A$ is isomorphic
  to the $C^*$-algebra of a row-finite graph with no
  sinks.
\item If $A$ is a simple AF-algebra, then $A$ is isomorphic
  to either an Exel-Laca algebra or a graph $C^*$
  -algebra. In particular, if $A$ is finite dimensional,
  then $A$ is isomorphic to a graph $C^*$-algebra; and if
  $A$ infinite dimensional, then $A$ is isomorphic to an
  Exel-Laca algebra.
\item If $A$ is an AF-algebra with no nonzero
  finite-dimensional quotients, then $A$ is isomorphic to
  an Exel-Laca algebra.
\end{itemize}

\noindent From our necessary conditions, we obtain the
following.
\begin{itemize}
\item  If an ultragraph $C^*$-algebra is a commutative
  AF-algebra then it is isomorphic to $c_0(X)$ for an at
  most countable discrete set $X$.
\item  No finite-dimensional $C^*$-algebra is isomorphic to an
  Exel-Laca algebra.
\item No infinite-dimensional UHF algebra is isomorphic to
  a graph $C^*$-algebra.
\end{itemize}

\noindent Moreover, we are able to give a characterization of
AF-algebras that are isomorphic to $C^*$-algebras of row-finite
graphs with no sinks in
Theorem~\ref{no-unital-quotient-then-graph-alg}.

\begin{theorem*}
Let $A$ be an AF-algebra. Then the following are equivalent:
\begin{enumerate}
\item $A$ has no unital quotients.
\item $A$ is isomorphic to the $C^*$-algebra of a
  row-finite graph with no sinks.
\end{enumerate}
\end{theorem*}

Our results allow us to make a fairly detailed analysis of the
AF-algebras in each of our three classes, and in
Figure~\ref{fig:Venn} at the end of this paper we draw a Venn
diagram relating various classes of AF-algebras among the graph
$C^*$-algebras, Exel-Laca algebras, and ultragraph
$C^*$-algebras.  Our results are powerful enough that we are
able to give examples in each region of the Venn diagram, and
also state definitively whether or not there are unital and
nonunital examples in each region.

Finally, we remark that a particularly useful aspect of our
sufficiency results is their constructive nature.  When one
first approaches the problem of identifying which AF-algebra
are in our three classes, one may be tempted to use the
$K$-theory classification of AF-algebras.  There are, however,
two problems with this approach: (1) Since any AF-algebra is
Morita equivalent to the $C^*$-algebra of a row-finite graph
with no sinks, we know that all ordered $K_0$-groups are
attained by the AF-algebras in each of our three classes.  Thus
we need to identify which \emph{scaled} ordered $K_0$-groups
are attained by the AF-algebras in each class. Unfortunately,
however, little is currently known about the scale for the
$K_0$-groups of $C^*$-algebras in these three classes. (2) More
importantly, even if we could decide exactly which scaled
ordered $K_0$-groups are attained by, for example, graph
AF-algebras, we would obtain at best an abstract
characterization of which AF-algebras are graph $C^*$-algebras.
Unless our understanding of the scaled ordered $K_0$-groups
achieved by AF graph $C^*$-algebras extended to an algorithm
for producing a graph whose $C^*$-algebra achieved a given
scaled ordered $K_0$-group, we would be unable to take a given
AF-algebra $A$ and view it as a graph $C^*$-algebra. Most
notably, we could not expect to ``see" the canonical generators
of $C^*(E)$ in $A$.

With an awareness of the limitations of an abstract
characterization, we instead present constructive methods for
realizing AF-algebras as $C^*$-algebras in our three classes.
Given a certain type of AF-algebra $A$ we show how to build an
ultragraph $\G$ from a certain type of Bratteli diagram for $A$
so that $C^*(\G)$ is isomorphic to $A$ (see
\S\ref{ultra-contruct-subsec}).  This ultragraph $C^*$-algebra is always an
Exel-Laca algebra, and in special situations (see
\S\ref{sufficient-cond-subsec}) it is also a graph
$C^*$-algebra.  Furthermore, one can extract from $\G$ a
$\{0,1\}$-matrix for the Exel-Laca algebra or a directed graph
for the graph $C^*$-algebra as appropriate.

This paper is organized as follows.  In \S\ref{prelim-sec} we
establish definitions and notation for graph $C^*$-algebras,
Exel-Laca algebras, ultragraph $C^*$-algebras, and AF-algebras.
In \S\ref{lemmas-sec} we establish some technical lemmas
regarding Bratteli diagrams and inclusions of
finite-dimensional $C^*$-algebras.  In \S\ref{results-sec} we
state the main results of this paper.  Specifically, in
\S\ref{ultra-contruct-subsec} we describe how to take a
Bratteli diagram for an AF-algebra $A$ with no nonzero
finite-dimensional quotients and build an ultragraph $\G$.  In
\S\ref{sufficient-cond-subsec} we prove that the associated
ultragraph $C^*$-algebra $C^*(\G)$ is isomorphic to $A$.  We
also show that $C^*(\G)$ is always isomorphic to an Exel-Laca
algebra, and describe conditions which imply  $C^*(\G)$ is also
a graph $C^*$-algebra.  These results give us a number of
sufficient conditions for AF-algebras to be contained in our
three classes of graph $C^*$-algebras, Exel-Laca algebras, and
ultragraph $C^*$-algebras.  We also present examples showing
that none of our sufficient conditions are necessary.  In
\S\ref{necessary-cond-subsec} we give several necessary
conditions for AF-algebras to be in each of our three classes.
These conditions allow us to identify a number of obstructions
to realizations of various AF-algebras in each class.  We
conclude in \S\ref{Venn-sec} by summarizing our containments.
First, we characterize precisely which simple AF-algebras fall
into each of our classes.  Second, we summarize many of the
relationships we have derived, including containments for the
finite-dimensional and stable AF-algebras, and draw a Venn
diagram to represent these containments.  We are able to use
our results from \S\ref{results-sec}  to exhibit examples in
each region of the Venn diagram, thereby showing these regions
are nonempty.  We are also able to describe precisely when
unital and nonunital examples occur in these regions.

\section{Preliminaries} \label{prelim-sec}

In the following four subsections we establish definitions and
notation for graph $C^*$-algebras, Exel-Laca algebras,
ultragraph $C^*$-algebras, and AF-algebras.  Since the
literature for each of these classes of $C^*$-algebras is large
and well developed, we present only the definitions and
notation required in this paper.  However, for each class we
provide introductory references where more detailed information
may be found.

\subsection{Graph $C^*$-algebras}  Introductory references include \cite{BPRS, Rae, Tom9}.

\begin{definition}
A \emph{graph} $E=(E^{0},E^{1},r,s)$ consists of a countable
set $E^{0}$ of vertices, a countable set $E^{1}$ of edges, and
maps $r \colon E^{1} \to E^{0}$ and $s \colon E^1 \to E^0$ identifying the
range and source of each edge.
\end{definition}

A \emph{path} in a graph $E = (E^0, E^1, r, s)$ is a sequence
of edges $\alpha := e_1 \ldots e_n$ with $s(e_{i+1}) = r(e_i)$
for $1 \leq i \leq n-1$. We say that $\alpha$ has \emph{length}
$n$.  We regard vertices as paths of length 0 and edges as
paths of length 1, and we then extend our notation for the
vertex set and the edge set by writing $E^n$ for the set of
paths of length $n$ for all $n \ge 0$. We write $E^*$ for the
set $\bigsqcup_{n=0}^\infty E^n$ of paths of finite length, and
extend the maps $r$ and $s$ to $E^*$ by setting $r(v) = s(v) =
v$ for $v \in E^0$, and $r(\alpha_1 \ldots \alpha_n) =
r(\alpha_n)$ and $s(\alpha_1\ldots\alpha_n) = s(\alpha_1)$.

If $\alpha$ and $\beta$ are elements of $E^*$ such that
$r(\alpha) = s(\beta)$, then $\alpha\beta$ is the path of
length $|\alpha|+|\beta|$ obtained by concatenating the two.
Given $\alpha, \beta \in E^*$, and a subset $X$ of $E^*$, we
let
\[
\alpha X \beta := \{ \gamma \in E^* : \gamma = \alpha \gamma' \beta \text{ for some } \gamma' \in X \}.
\]
So when $v$ and $w$ are vertices, we have
\begin{align*}
vX &= \{\gamma \in X : s(\gamma) = v\},\\
Xw &= \{\gamma \in X : r(\gamma) = w\},\text{ and}\\
vXw &= \{ \gamma \in X : s(\gamma) = v \text{ and } r(\gamma) = w \}.
\end{align*}
In particular, $vE^1w$ denotes the set of edges from $v$ to $w$
and $|vE^1w|$ denotes the number of edges from $v$ to $w$.

We say a vertex $v$ is a \emph{sink} if $vE^1 = \emptyset$ and
an \emph{infinite emitter} if $vE^1$ is infinite. A graph is
called \emph{row-finite} if it has no infinite emitters.

\begin{definition}[Graph $C^*$-algebras]
If $E=(E^{0},E^{1},r,s)$ is a graph, then the \emph{graph
$C^*$-algebra} $C^{*}(E)$ is the universal $C^*$-algebra
generated by mutually orthogonal projections $\{p_{v} : v \in
E^{0} \}$ and partial isometries $\{ s_{e} : e \in E^{1} \}$
with mutually orthogonal ranges satisfying
\begin{enumerate}
\item $s_{e}^{*}s_{e} = p_{r(e)}$ \ \ for all $e \in E^{1}$
\item $p_{v}=\displaystyle  {\ \sum_{e \in vE^1} s_{e}
   s_{e}^*}$ \ \ for all $v\in E^{0}$ such that
   $0<|vE^1|<\infty$
\item $s_{e}s_{e}^{*} \leq p_{s(e)}$ \ \ for all $e \in
   E^{1}$.
\end{enumerate}
\end{definition}

We write $v \geq w$ to mean that there is a path $\alpha \in
E^*$ such that $s(\alpha) = v$ and $r(\alpha) = w$.
A \emph{cycle} in a graph $E$ is a path $\alpha \in E^*$ of
nonzero length with $r(\alpha) = s(\alpha)$.
\cite[Theorem~2.4]{KPR} says that $C^*(E)$ is an AF-algebra
if and only if $E$ has no cycles.

\subsection{Exel-Laca algebras} Introductory references include \cite{EL, EL2, FLR, Szy4}.

\begin{definition}[Exel-Laca algebras]  Let $I$ be a finite or countably infinite set, and let $A =
\{A(i,j)\}_{i,j \in I}$ be a $\{0,1\}$-matrix over $I$ with no
identically zero rows.  The \emph{Exel-Laca algebra}
$\mathcal{O}_A$ is the universal $C^*$-algebra generated by
partial isometries $\{ s_i : i \in I \}$ with commuting initial
projections and mutually orthogonal range projections
satisfying $s_i^* s_i s_j s_j^* = A(i,j) s_js_j^*$ and
\begin{equation}
\label{conditionfour}
\prod_{x \in X} s_x^*s_x \prod_{y \in Y} (1-s_y^*s_y) =
\sum_{j \in I} A(X,Y,j) s_js_j^*
\end{equation}
whenever $X$ and $Y$ are finite subsets of $I$ such that $X
\not= \emptyset$ and the function $$j \in I \mapsto A(X,Y,j) :=
\prod_{x \in X} A(x,j) \prod_{y \in Y} (1-A(y,j))$$ is finitely
supported.  (We interpret the unit in (\ref{conditionfour}) as
the unit in the multiplier algebra of $\mathcal{O}_A$.)
\end{definition}

We will see in Remark~\ref{rem:ELalgis,,,}
that for a $\{0,1\}$-matrix $A$ with no identically zero rows,
the canonical ultragraph $\G_A$ of $A$
satisfies $C^*(\G_A) \cong \mathcal{O}_A$.
With this notation, \cite[Theorem~4.1]{Tom2} implies
that the Exel-Laca algebra $\mathcal{O}_A$ is an AF-algebra
if and only if $\G_A$ has no cycle.
The latter condition can be restated as:
there does not exist a finite set $\{i_1, \ldots, i_n\} \subseteq I$
with $A(i_k, i_{k+1}) = 1$ for all $1 \leq k \leq n-1$ and
$A(i_n, i_1) =1$.

It is well known that the class of graph $C^*$-algebras of
row-finite graphs with no sinks and the class of Exel-Laca
algebras of row-finite matrices coincide.  However, we have
been unable to find a reference, so we give a proof here.

\begin{lemma} \label{row-finite-graphs-matrices-lem}
The class of graph $C^*$-algebras of row-finite graphs with no
sinks and the class of Exel-Laca algebras of row-finite
matrices coincide.  In particular,
\begin{enumerate}
\item If $E = (E^0, E^1, r, s)$ is a row-finite graph with
  no sinks, and if we define a $\{0,1\}$-matrix $A_E$ over $E^1$
  by $$A_E(e,f) := \begin{cases} 1 & \text{ if r(e)
  = s(f)} \\ 0 & \text{ otherwise,}\end{cases}$$ then
  $A_E$ is a row-finite matrix with no identically zero rows
  and $C^*(E) \cong \mathcal{O}_{A_E}$.
\item If $A$ is a row-finite $\{0,1\}$-matrix over $I$ with no
  identically zero rows, and if we define a graph $E_A$ by
  setting $E_A^0 := I$ and drawing an edge from $v \in I$
  to $w \in I$ if and only if $A(v,w) = 1$, then
  $E_A$ is a row-finite graph with no sinks
  and $\mathcal{O}_A \cong C^*(E_A)$.
\end{enumerate}
\end{lemma}

\begin{proof}
For (1) let $E = (E^0, E^1, r, s)$ be a row-finite graph with
no sinks, and define the matrix $A_E$ as above.
Since $E$ is row-finite, $A_E$ is also row-finite.
Let $\{ S_e : e \in E^1 \}$ be a generating Exel-Laca
$A_E$-family in $\mathcal{O}_{A_E}$.  For $v \in E^0$ we define
$P_v := \sum_{s(e) = v} S_eS_e^*$ in $\mathcal{O}_{A_E}$.
(Note that this sum is always finite since $A_E$ is row-finite.)
We now show that $\{S_e, P_v : e \in E^1, v \in E^0 \}$
is a Cuntz-Krieger $E$-family in $\mathcal{O}_{A_E}$.
The $S_e$'s have mutually orthogonal
range projections by the Exel-Laca relations, and hence the
$P_v$'s are also mutually orthogonal projections.  In addition,
Condition~(2) and Condition~(3) in the definition of graph
$C^*$-algebras obviously hold from our definition of $P_v$.  It
remains to show Condition~(1) holds.  If $e \in E^1$, let $X :=
\{ e \}$ and $Y := \emptyset$.  Then for $j \in E^1$, we have
$A_E(X,Y,j) := 1$ if and only if $s(j) = r(e)$.  Since $E$ is
row-finite, the function $j \mapsto A_E(X,Y,j)$ is finitely
supported, and (\ref{conditionfour}) gives $S_e^*S_e = \sum_{j
\in E^1} A(X,Y,j) S_jS_j^* = \sum_{ s(j) = r(e) } S_jS_j^* =
P_{r(e)}$, so Condition~(1) holds.  Thus $\{S_e, P_v : e \in
E^1, v \in E^0 \}$ is a Cuntz-Krieger $E$-family, and by the
universal property of $C^*(E)$ we obtain a \shom
$\phi \colon C^*(E) \to \mathcal{O}_{A_E}$ with $\phi(s_e) = S_e$
and $\phi(p_v) = P_v$ where $\{s_e, p_v \}$ is a generating
Cuntz-Krieger $E$-family for $C^*(E)$.
By checking on generators, one can see
that $\phi$ is equivariant with respect to the gauge actions on
$C^*(E)$ and $\mathcal{O}_{A_E}$, and thus the Gauge-Invariant
Uniqueness Theorem \cite[Theorem~2.1]{BPRS} implies that $\phi$
is injective.  Since the image of $\phi$ contains the
generators $\{ S_e : e \in E^1 \}$ of $\mathcal{O}_{A_E}$,
$\phi$ is also surjective.  Thus $C^*(E) \cong
\mathcal{O}_{A_E}$.

For (2) let $A$ be a row-finite $\{0,1\}$-matrix with no
identically zero rows.
Let $\G_A$ be the canonical ultragraph of $A$
(see Remark~\ref{rem:ELalgis,,,}).
Then the source map of $\G_A$ is bijective and
$C^*(\G_A) \cong \mathcal{O}_A$.
Since $A$ is a row-finite matrix,
the range of each edge in $\G_A$ is a finite set.
Thus $C^*(\G_A)$ is isomorphic to the $C^*$-algebra of the graph
formed by replacing each edge in $\G_A$ with a set of edges
from $s(e)$ to $w$ for all $w \in r(e)$ \cite[Remark~2.5]{KMST2}.
But this is precisely the graph $E_A$ described in the statement above.
\end{proof}

\subsection{Ultragraph $C^*$-algebras}
Introductory references include \cite{KMST, KMST2, Tom, Tom2}.
For a set $X$, let $\Pow(X)$ denote the collection of all subsets of $X$.

\begin{definition}(\cite[Definition~2.1]{Tom})
An \emph{ultragraph} $\G = (G^0, \G^1, r,s)$ consists of a
countable set of vertices $G^0$, a countable set of ultraedges
$\G^1$, and functions $s\colon \G^1\rightarrow G^0$ and
$r\colon \G^1 \rightarrow \Pow(G^0)\setminus\{\emptyset\}$.
\end{definition}

Note that in the literature, ultraedges are typically just
referred to as edges. However, since we will frequently be
passing back and forth between graphs and ultragraphs in this
paper, we feel that using the term ultraedge will serve as a
helpful reminder that edges in ultragraphs behave differently
than in graphs.

\begin{definition}
For a set $X$, a subset $\cC\/$ of $\Pow(X)$ is called an
{\em algebra} if
\begin{enumerate}
\item[(i)] $\emptyset\in\cC$,
\item[(ii)] $A\cap B\in\cC$ and $A\cup B\in\cC$ for all
  $A,B\in\cC$, and
\item[(iii)] $A\setminus B\in \cC$ for all $A,B\in\cC$.
\end{enumerate}
\end{definition}

\begin{definition}
For an ultragraph $\G = (G^0, \G^1, r,s)$, we let $\G^0$ denote
the smallest algebra in $\Pow(G^0)$ containing the singleton
sets and the sets  $\{r(e) : e \in \G^1\}$.
\end{definition}

\begin{definition}
A \emph{representation} of an algebra $\cC$ is
a collection of projections $\{p_A\}_{A\in \cC}$ in a \Ca
satisfying $p_\emptyset = 0$, $p_A p_B = p_{A \cap B}$, and
$p_{A\cup B} = p_A + p_B - p_{A \cap B}$ for all $A,B \in \cC$.
\end{definition}

Observe that a representation of an algebra automatically
satisfies $p_{A \setminus B} = p_A - p_A p_B$.

\begin{definition} \label{dfn:CK-G-fam}
For an ultragraph $\G = (G^0, \G^1, r,s)$, the \emph{ultragraph
$C^*$-algebra} $C^*(\G)$ is the universal $C^*$-algebra
generated by a representation $\{p_A\}_{A\in \G^0}$ of $\G^0$
and a collection of partial isometries $\{s_e\}_{e \in \G^1}$
with mutually orthogonal ranges that satisfy
\begin{enumerate}
\item $s_e^*s_e = p_{r(e)}$ for all $e \in \G^1$,
\item $s_es_e^* \leq p_{s(e)}$ for all $e \in \G^1$,
\item $p_v = \sum_{e \in v\G^1} s_es_e^*$ whenever $0 <
   |v\G^1| < \infty$,
\end{enumerate}
where we write $p_v$ in place of $p_{ \{ v \} }$ for
$v \in G^0$.
\end{definition}

As with graphs, we call a vertex $v \in G^0$ a \emph{sink} if
$v\G^1 = \emptyset$ and an \emph{infinite emitter} if $v\G^1$
is infinite.  A \emph{path} in an ultragraph $\G$ is a sequence
of ultraedges $\alpha = e_1 e_2 \ldots e_n$ with $s(e_{i+1})
\in r(e_i)$ for $1 \leq i \leq n-1$.  A \emph{cycle} is a path
$\alpha = e_1 \ldots e_n$ with $s(e_1) \in r(e_n)$.
\cite[Theorem~4.1]{Tom2} implies that $C^*(\G)$ is an
AF-algebra if and only if $\G$ has no cycles.

\begin{remark}\label{rem:ELalgis,,,}
A graph may be regarded as an ultragraph in which the range of
each ultraedge is a singleton set. The constructions of the two
$C^*$-algebras then coincide: the graph $C^*$-algebra of a
graph is the same as the ultragraph $C^*$-algebra of that graph when
regarded as an ultragraph (see \cite[\S3]{Tom} for more details).

For a $\{0,1\}$-matrix $A$ over $I$ with nonzero rows,
the canonical associated ultragraph $\G_A=(G_A^0, \G_A^1, r,s)$
is defined by $G_A^0 = \G_A^1 = I$,
$r(i)=\{j\in I : A(i,j)=1\}$ and $s(i)=i$ for $i \in \G_A^1$
(see \cite[Definition~2.5]{Tom}).
It follows from \cite[Theorem~4.5]{Tom} that
$C^*(\G_A) \cong \mathcal{O}_A$.
The ultragraph $\G_A$ has the property that
$s$ is bijective.
Conversely an ultragraph $\G = (G^0, \G^1, r, s)$ with bijective $s$
is isomorphic to $\G_A$
where $A$ is the edge matrix of $\G$.
Thus one can say that an Exel-Laca algebra
is a $C^*$-algebra of a ultragraph
with bijective source map.

From these observations,
one can see that the class of ultragraph $C^*$-algebras
contains both the class of graph $C^*$-algebras and the class
of Exel-Laca algebras.
\end{remark}

\subsection{AF-algebras}

Introductory references include \cite{Bra, Eff, GH} as well as
\cite[Chapter 6]{Dav} and \cite[\S6.1, \S6.2, and \S7.2]{Mur}.

\begin{definition}
An AF-algebra is a $C^*$-algebra that is the direct limit of a
sequence of finite-dimensional $C^*$-algebras.  Equivalently, a $C^*$-algebra $A$ is an AF-algebra
if and only if $A = \overline{\bigcup_{n=1}^\infty A_n}$ for a
sequence of finite-dimensional $C^*$-subalgebras $A_1 \subseteq
A_2 \subseteq \cdots \subseteq A$.
\end{definition}

To discuss AF-algebras, we need first to briefly discuss
inclusions of finite-dimensional $C^*$-algebras. Fix
finite-dimensional $C^*$-algebras $A = \bigoplus^m_{i=1}
M_{a_i}(\C)$ and $B = \bigoplus^n_{j=1} M_{b_j}(\C)$. Let
$M=(m_{i,j})_{i,j}$ be an $m \times n$ nonnegative integer
matrix with no zero rows such that
\begin{equation}\label{eq:dimensions fit}
\sum^m_{i=1} m_{i,j} a_j \le b_j\quad\text{for all $j$}.
\end{equation}
There exists an inclusion $\phi_M \colon A \hookrightarrow B$ with
the following property. For an element $x = (x_i)^m_{i=1} \in
A$, the image $\phi_M(x)$ of $x$ has the form $(y_j)^n_{j=1}
\in B$ where for each $j \le n$, the matrix $y_j$ is
block-diagonal with $m_{i,j}$ copies of each $x_i$ along the
diagonal and $0$'s elsewhere. (Equation~\ref{eq:dimensions fit}
ensures that this is possible.) The map $\phi_M$ is not
uniquely determined by this property, but its unitary
equivalence class is.

Every inclusion $\phi$ of $A$ into $B$ is unitarily equivalent to
$\phi_M$ for some matrix $M$.
Specifically, $M=(m_{i,j})_{i.j}$ is the matrix
such that $m_{i,j}$ is equal to the rank of $1_{B_j}
\phi(p_i)$ where $1_{B_j}$ is the unit for the
$j$\textsuperscript{th} summand of $B$, and where $p_i$ is any
rank-1 projection in the $i$\textsuperscript{th} summand of
$A$. We refer to $M$ as the \emph{multiplicity matrix} of the
inclusion $\phi$.

\begin{definition} \label{Brat-diag-defn}
A \emph{Bratteli diagram} $(E,d)$ consists of a directed graph
$E = (E^0,E^1,r,s)$ together with a collection $d = \{ d_v : v
\in E^0 \}$ of positive integers satisfying the following
conditions.
\begin{enumerate}
\item[(1)] $E$ has no sinks;
\item[(2)] $E^0$ is partitioned as a disjoint union $E^0 =
  \bigsqcup_{n=1}^\infty V_n$ where each $V_n$ is a
  finite set,
\item[(3)] for each $e \in E^1$ there exists $n \in \N$
  such that $s(e) \in V_n$ and $r(e) \in V_{n+1}$; and
\item[(4)]  for each vertex $v \in E^0$ we have $d_v \geq
   \sum_{e \in E^1v} d_{s(e)}$ for all $v \in E^0$.
\end{enumerate}
\end{definition}

If $(E,d)$ is a Bratteli diagram, then $E$ is a row-finite
graph with no sinks. We regard $d$ as a labeling of the
vertices by positive integers, so to draw a Bratteli diagram we
sometimes just draw the directed graph, replacing each vertex
$v$ by its label $d_v$.

\begin{remark}
Those experienced with Bratteli diagrams will notice that our
definition of a Bratteli diagram is slightly nonstandard.
Specifically, a Bratteli diagram is traditionally specified as
undirected graph in which each edge connects vertices in
consecutive levels.  Of course, an orientation of the edges is
then implicitly chosen by the decomposition $E^0 = \bigsqcup
V_n$, so it makes no difference if we instead draw a directed
edge pointing from the vertex in level $n$ to the vertex in
level $n+1$.
\end{remark}

\begin{example} \label{Bratteli-Ex-one}
The following is an example of a Bratteli diagram:
%
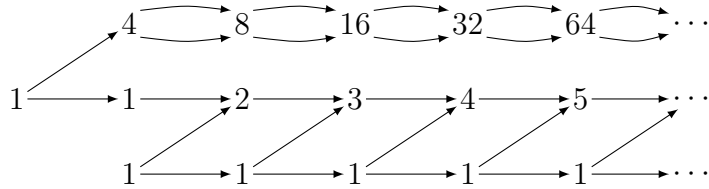
\[\begin{tikzpicture}[xscale=1.5]
   \node[inner sep=1pt] (b1) at (1,0) {1};%
   \node[inner sep=1pt] (a2) at (2,1) {4};%
   \node[inner sep=1pt] (b2) at (2,0) {1};%
   \node[inner sep=1pt] (c2) at (2,-1) {1};%
   \node[inner sep=1pt] (a3) at (3,1) {8};%
   \node[inner sep=1pt] (b3) at (3,0) {2};%
   \node[inner sep=1pt] (c3) at (3,-1) {1};%
   \node[inner sep=1pt] (a4) at (4,1) {16};%
   \node[inner sep=1pt] (b4) at (4,0) {3};%
   \node[inner sep=1pt] (c4) at (4,-1) {1};%
   \node[inner sep=1pt] (a5) at (5,1) {32};%
   \node[inner sep=1pt] (b5) at (5,0) {4};%
   \node[inner sep=1pt] (c5) at (5,-1) {1};%
   \node[inner sep=1pt] (a6) at (6,1) {64};%
   \node[inner sep=1pt] (b6) at (6,0) {5};%
   \node[inner sep=1pt] (c6) at (6,-1) {1};%
   \node[inner sep=1pt] (a7) at (7,1) {$\cdots$};%
   \node[inner sep=1pt] (b7) at (7,0) {$\cdots$};%
   \node[inner sep=1pt] (c7) at (7,-1) {$\cdots$};%
   \draw[-latex] (b1)--(b2);%
   \draw[-latex] (b1)--(a2);%
   \foreach \x/\xx in {2/3,3/4,4/5,5/6,6/7} {%
       \draw[-latex] (a\x.north east) .. controls (\x.5,1.25) .. (a\xx.north west);%
       \draw[-latex] (a\x.south east) .. controls (\x.5,0.75) .. (a\xx.south west);%
       \draw[-latex] (b\x)--(b\xx);%
       \draw[-latex] (c\x)--(b\xx);%
       \draw[-latex] (c\x)--(c\xx);%
   }%
\end{tikzpicture}\]
\end{example}

\vskip1ex

Given a Bratteli diagram $(E,d)$, we construct an AF-algebra $A$ as
follows. For each $v \in E^0$, let $A_v$ be an isomorphic copy
of $M_{d_v}(\C)$, and for each $n \in \N$, let $A_n :=
\bigoplus_{v \in V_n} A_v$. For each $n$ let $\phi_n \colon A_n \to
A_{n+1}$ be the homomorphism whose multiplicity matrix is
$( |vE^1w| )_{v \in V_n, w \in V_{n+1}}$.
We then define $A$ to be the direct limit
$\varinjlim(A_n,\phi_n)$. Since the $\phi_n$ are determined up
to unitary equivalence by $(E,d)$, the isomorphism class of $A$ is
also uniquely determined by $(E,d)$.

\begin{example}
In Example~\ref{Bratteli-Ex-one}, we see that
\begin{align*}
A_1 & = \C & \phi_1(x) & = \left( \begin{smallmatrix} x & 0 & 0 & 0 \\ 0 & 0 & 0 & 0 \\ 0 & 0 & 0 & 0 \\ 0 & 0 & 0 & 0 \end{smallmatrix} \right) \oplus x \oplus 0 \\
A_2 &= M_4 (\C) \oplus \C \oplus \C & \phi_2(x, y, z) &= \left( \begin{smallmatrix} x & 0 \\ 0 & x \end{smallmatrix} \right) \oplus  \left( \begin{smallmatrix} y & 0 \\
0 & z \end{smallmatrix} \right) \oplus z \\
A_3 &= M_8(\C) \oplus M_2(\C) \oplus \C & \phi_3(x, y, z) &= \left( \begin{smallmatrix} x & 0 \\ 0 & x \end{smallmatrix} \right) \oplus  \left( \begin{smallmatrix} y & 0 \\
0 & z \end{smallmatrix} \right) \oplus z \\
& \ \vdots & & \ \vdots \\
A_n &= M_{2^{n}} (\C) \oplus M_{n-1} (\C) \oplus \C & \phi_n(x, y, z) &= \left( \begin{smallmatrix} x & 0 \\ 0 & x \end{smallmatrix} \right) \oplus  \left( \begin{smallmatrix} y & 0 \\ 0 & z \end{smallmatrix} \right) \oplus z \\
&\ \vdots & & \ \vdots
\end{align*}
\end{example}
\smallskip

The following \emph{telescoping} operation on a Bratteli
diagram preserves the associated AF-algebra.
Given $(E,d)$, we choose an increasing subsequence $\{n_m\}_{m=1}^\infty$
of $\N$.
The set of the vertices of the new Bratteli diagram
is $\bigcup_{m=1}^\infty V_{n_m}$,
the set of the edges of the new Bratteli diagram is
$\bigcup_{m=1}^\infty (V_{n_m}E^*V_{n_{m+1}})$,
and the new function $d$ is the restriction of the old
$d$ to $\bigcup_{m=1}^\infty V_{n_m}$.
For example, if we have the
portion of a Bratteli diagram shown below on the left and
remove the middle column of vertices, we obtain the portion of
the Bratteli diagram shown below on the right.
%
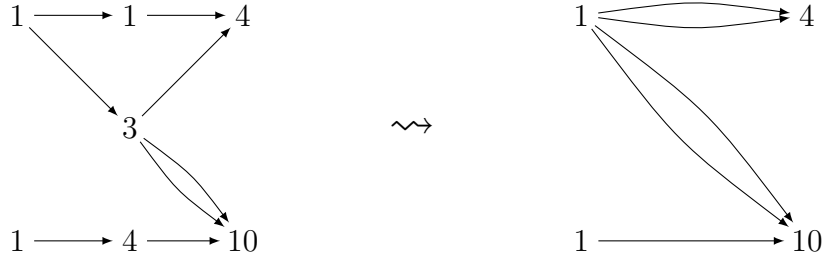
\[\begin{tikzpicture}[scale=1.5]
   \node[circle,inner sep=1pt] (a0) at (0,2) {1};
   \node[circle,inner sep=1pt] (c0) at (0,0) {1};
   \node[circle,inner sep=1pt] (a1) at (1,2) {1};
   \node[circle,inner sep=1pt] (b1) at (1,1) {3};
   \node[circle,inner sep=1pt] (c1) at (1,0) {4};
   \node[circle,inner sep=1pt] (a2) at (2,2) {4};
   \node[circle,inner sep=1pt] (c2) at (2,0) {10};
   \draw[-latex] (a0)--(a1);
   \draw[-latex] (a0)--(b1);
   \draw[-latex] (c0)--(c1);
   \draw[-latex] (a1)--(a2);
   \draw[-latex] (b1)--(a2);
   \draw[color=white] (b1)--(c2) node[pos=0.5,inner sep=0.1cm] (b1c2) {};
   \draw[-latex] (b1) .. controls (b1c2.north east) .. (c2);
   \draw[-latex] (b1) .. controls (b1c2.south west) .. (c2);
   \draw[-latex] (c1)--(c2);
   \node at (3.5,1) {\Large$\rightsquigarrow$};
   \node[circle,inner sep=1pt] (x0) at (5,2) {1};
   \node[circle,inner sep=1pt] (y0) at (5,0) {1};
   \node[circle,inner sep=1pt] (x2) at (7,2) {4};
   \node[circle,inner sep=1pt] (y2) at (7,0) {10};
   \draw[color=white] (x0)--(x2) node[pos=0.5,inner sep=0.2cm] (x0x2) {};
   \draw[-latex] (x0) .. controls (x0x2.north) .. (x2);
   \draw[-latex] (x0) .. controls (x0x2.south) .. (x2);
   \draw[color=white] (x0)--(y2) node[pos=0.5,inner sep=0.2cm] (x0y2) {};
   \draw[-latex] (x0) .. controls (x0y2.north east) .. (y2);
   \draw[-latex] (x0) .. controls (x0y2.south west) .. (y2);
   \draw[-latex] (y0)--(y2);
\end{tikzpicture}\]

\vskip1ex

We say that two Bratteli diagrams $(E,d)$ and $(E',d')$ are
equivalent if there is a finite sequence $(E_1, d_1),
\ldots, (E_n, d_n)$ such that $(E_1, d_1) = (E,d)$, $(E_n, d_n)
= (E', d')$ and for each $1 \le i \le n-1$, one of $(E_i, d_i)$
and $(E_{i+1}, d_{i+1})$ is a telescope of the other. Bratteli
proved in \cite{Bra} that two Bratteli diagrams give rise to
isomorphic AF-algebras if and only if the diagrams are
equivalent  (see \cite[\S 1.8]{Bra} and \cite[Theorem~2.7]{Bra}
for details).

The class of AF-algebras is closed under forming ideals and
quotients. On the other hand, the three classes of graph
$C^*$-algebras, Exel-Laca algebras, and ultragraph
$C^*$-algebras are not closed under forming ideals nor
quotients. However we can show the following.

\begin{lemma}\label{ideal_quotient}
The class of graph AF-algebras is closed under forming ideals
and quotients.
\end{lemma}

\begin{proof}
If $E$ is a graph and the graph $C^*$-algebra $C^*(E)$ is an
AF-algebra, then $E$ has no cycles by \cite[Theorem~2.4]{KPR}.
Thus $E$ vacuously satisfies Condition~(K),
and it follows that every ideal of $C^*(E)$ is gauge-invariant
by \cite[Corollary~3.8]{BHRS}.
Thus every ideal of $C^*(E)$ as well as its quotient
is a graph $C^*$-algebra
by \cite[Lemma~1.6]{DHS} and \cite[Theorem~3.6]{BHRS}.
\end{proof}

\begin{remark}
A quotient of an Exel-Laca AF-algebra
need not be an Exel-Laca algebra.
For example, if $\MU{\K}$
is the minimal unitization of the compact operators $\K$
on a separable infinite-dimensional Hilbert space, then
$M_2(\MU{\K})$ is an Exel-Laca AF-algebra that has a quotient,
$M_2(\C)$, that is not an Exel-Laca algebra --- for details see
Example~\ref{E-L-quotient-not-Ex} and
Corollary~\ref{No-fin-dim-E-L-Cor}.
Whether ideals of Exel-Laca AF-algebras
are necessarily Exel-Laca algebras is an open question.
We also do not know
whether ideals and quotients of
ultragraph AF-algebras are necessarily ultragraph
$C^*$-algebras.
As we shall see later, this uncertainty causes
problems in the analyses of Exel-Laca AF-algebras
and ultragraph AF-algebras.
\end{remark}

\begin{lemma}
The three classes of graph AF-algebras, Exel-Laca AF-algebras, and ultragraph AF-algebras are closed
under taking direct sums.
\end{lemma}
\begin{proof}
Each of the four classes of AF-algebras, graph $C^*$-algebras,
Exel-Laca algebras, and ultragraph $C^*$-algebras is closed
under forming direct sums. The result follows.
\end{proof}

\section{Some technical lemmas} \label{lemmas-sec}

In this section we establish some technical results for
Bratteli diagrams and inclusions of finite-dimensional
$C^*$-algebras.  We will use these technical results to prove
many of our realization results in \S\ref{results-sec}.

\begin{lemma} \label{fin-quotient-lem-1}
Suppose $A$ is an AF-algebra that has no quotients isomorphic
to $\C$, and suppose that $(E,d)$ is a Bratteli diagram for
$A$. Let $H = \{ v \in E^0 : d_v = 1 \}$, and let $F$ be the
subgraph of $E$ such that $F^0 := E^0 \setminus H$
and $F^1 := \{e \in E^1 : s(e) \not\in H\}$ with $r, s \colon F^1 \to F^0$
inherited from $E$. Let $d \colon F^0 \to \N$ be the
restriction of $d \colon E^0 \to \N$. Then $(F,d)$ is a
Bratteli diagram for $A$.
\end{lemma}

\begin{proof}
First note that if $e \in E^1$ with $r(e) \in H$, then
$d_{r(e)}=1$ and hence $d_{s(e)} = 1$ and $s(e) \in H$.  Hence
$F$ is in fact a subgraph of $E$.

We claim that for any $n \in \N$ and $v \in V_n$, there exists
$m \in \N$ such that whenever $w \in V_{n+m}$ and $v \geq w$,
we have $d_w \geq 2$.  We fix $n \in \N$ and $v \in V_n$,
suppose that there is no such $m$, and seek a contradiction.
Let $v_0 := v$. Inductively choose $e_i \in E^1$ such that
$s(e_i) = v_{i-1}$ and such that for each $m \in \N$ there
exists $w \in V_{n+i+m}$ with $r(e_i) \ge w$ and $d_w = 1$,
setting $v_i := r(e_i)$. Then the infinite path $e_1e_2 \ldots$
satisfies $d_{s(e_n)} = 1$ for all $n$. Hence $\{ x \in E^0 : x
\not\geq s(e_n) \text{ for any $n$} \}$ is a saturated
hereditary subset and the quotient of $A$ by the corresponding
ideal is an AF-algebra with Bratteli diagram
\[\begin{tikzpicture}[xscale=1.5]
   \foreach \x in {1,2,3,4,5} {
       \node[inner sep=1pt] (\x) at (\x,0) {$\ 1 \ $};
   }
   \node (6) at (6,0) {$\cdots$};
   \foreach \x/\xx in {1/2,2/3,3/4,4/5,5/6} {
       \draw[-latex] (\x)--(\xx);
   }
\end{tikzpicture}\]
Hence this quotient is isomorphic to $\C$, which contradicts
our hypothesis on $A$.  This establishes the claim.

Let $B$ be the AF-algebra associated to the Bratteli diagram
$F$, and let $\iota_n \colon B_n \to A_n$ denote obvious inclusion
of the $n$\textsuperscript{th} approximating subalgebra of $B$
determined by $F$ into the $n$\textsuperscript{th}
approximating subalgebra of $A$ determined by $E$.  Let
$\phi^E_{n,m} \colon A_n \to A_m$ be the connecting maps in the
directed system associated to $E$, and let $\phi^E_{n, \infty}
: A_n \to A$ be the inclusion of $A_n$ into the direct limit
algebra $A$.  Likewise, let $\phi^F_{n,m} \colon B_n \to B_m$ be the
connecting maps in the directed system associated to $F$, and
let $\phi^F_{n, \infty} \colon B_n \to B$ be the inclusion of $B_n$
into the direct limit algebra $B$.

We see that $\phi_{n,n+1}^E \circ \iota_n = \iota_{n+1} \circ
\phi_{n,n+1}^F$ for all $n$, and thus by the universal property
of the direct limit $B = \varinjlim(B_n, \phi^F_n)$, there is a
\shom $\iota_\infty \colon B \to A$ with $\phi_{n, \infty}^E \circ
\iota_n = \iota_\infty \circ \phi_{n, \infty}^F$.  Since each
$\iota_n$ is injective, it follows that $\iota_\infty$ is
injective.  We shall also show that $\iota_\infty$ is also
surjective and hence an isomorphism. It suffices to show that
for any $v \in V_n$ and for any $a$ in the direct summand $A_v$
of $A_n$ corresponding to $v$, we have $\phi^E_{n,\infty}(a)
\in \operatorname{im} \iota_\infty$. By the previous paragraph
we may choose $m$ so that whenever $w \in V_{n+m}$ and $v \geq
w$, then $d_w \geq 2$.  It follows that
\[
\phi^E_{n, n+m}(a) \in \bigoplus_{\substack{w \in
V_{n+m} \\ d_w \geq 2}} M_{d_w} (\C) \subseteq \iota_{n+m}
(B_{n+m}),
\]
so that $\phi^E_{n, n+m}(a) = \iota_{n+m} (b)$ for
some $b \in B_{n+m}$.  Thus
\[
\phi^E_{n,\infty} (a) = \phi^E_{n+m,
\infty} \circ \phi^E_{n, n+m}(a) = \phi^E_{n+m, \infty} \circ
\iota_{n+m} (b) = \iota_\infty \circ \phi^F_{n+m, \infty} (b)
\in \operatorname{im} \iota_\infty
\]
and $\iota_\infty$ is surjective. Hence $\iota_\infty$ is an isomorphism
as required.
\end{proof}

\begin{lemma} \label{fin-quotient-lem-2}
Suppose $A$ is an AF-algebra with no nonzero finite-dimensional
quotients.  Then any Bratteli diagram for $A$ can be telescoped
to obtain a second Bratteli diagram $(E,d)$ for $A$ such that
for all $n \in \N$ and for each $v \in V_{n+1}$ either $d_v >
\sum_{e \in E^1v} d_{s(e)}$ or there exists $w \in V_n$ with
$|wE^1v| \geq 2$.
\end{lemma}

\begin{proof}
Let $(F,d)$ be a Bratteli diagram for $A$ with $F^0$
partitioned into levels as $F^0 = \bigsqcup_{n=1}^\infty W_n$.
It suffices to show that for every $m$ there exists $n \ge m$
such that for every $v \in W_n$ satisfying $d_v = \sum_{\alpha
\in W_m F^* v} d_{s(\alpha)}$, there exists $w \in W_m$ with
$|w F^* v| \ge 2$.
We suppose not, and seek a contradiction. That is, we suppose
that there exists $m$ such that for every $n \ge m$ the set
\[
X_n := \Big\{x \in W_n : d_x = \sum_{\alpha \in W_m F^* x} d_{s(\alpha)}\text{ and } |wF^* x| \le 1\text{ for all }w \in W_m\Big\}
\]
is nonempty.
By telescoping $(F,d)$ to $\bigsqcup_{n=m}^\infty W_n$
we may assume $m=1$.
We claim that if $n \le p$,
$x \in X_p$, and $v \in W_n$ with $v \ge x$, then $v \in X_n$.
Indeed,
\begin{align}
d_x
	&= \sum_{\alpha \in W_1 F^* x} d_{s(\alpha)}\notag \\
	&= \sum_{\beta \in W_n F^* x}\Big(\sum_{\gamma \in W_1 F^* s(\beta)} d_{s(\gamma)}\Big) \label{line1} \\
	&\le \sum_{\beta \in W_n F^* x} d_{s(\beta)} \label{line2} \\
	&\le d_x. \notag
\end{align}
Thus we have equality throughout, and
the equality of \eqref{line1}~and~\eqref{line2} implies
$d_{s(\beta)} = \sum_{\gamma \in W_1 F^* s(\beta)}
d_{s(\gamma)}$ for each $\beta \in W_n F^* x$.
In particular, since $v \ge x$,
we have that $d_v = \sum_{\gamma \in W_1 F^* v} d_{s(\gamma)}$.
Moreover for each $w \in W_1$,
\[
1 \ge |w F^* x| \ge |w F^* v| \, |v F^* x|,
\]
so $v \ge x$ implies that $|w F^* v| \le 1$,
and $v \in X_n$ as required.

We shall now construct an infinite path $\lambda = \lambda_1
\lambda_2 \ldots$ in $F$ such that $s(\lambda_n) \in X_{n}$ for
all $n$. If $x \in X_{n}$, then since $d_x$ is nonzero and $d_x
= \sum_{\alpha \in W_1 F^* x} d_{s(\alpha)}$, there exists $w
\in W_1$ such that $w \ge x$. Since $W_1$ is finite, there
exists $w_1 \in W_1$ such that for infinitely many $n$ there
exists $x \in X_{n}$ with $w_1 \geq x$.  Since $w_1F^1$ is
finite, there exists $\lambda_1 \in w_1F^1$ such that for
infinitely many $n$, we have $r(\lambda_1) \ge x$ for some $x
\in X_{n}$. We set $w_2 := r(\lambda_1)$ which is in $X_{2}$ by
the claim above. Proceeding in this way, we produce an infinite
path $\lambda = \lambda_1 \lambda_2 \ldots$ in $F$ such that
$s(\lambda_n) \in X_{n}$ for all $n$.

For each $w \in W_1$ such that $w \ge s(\lambda_n)$ for some
$n$, we define $n_w := \min\{n : w \ge s(\lambda_n)\}$. Let $N
:= \max\{n_w : \text{$w \in W_1$ and $w \ge s(\lambda_n)$ for
some $n$}\}$. We claim that $F^1 r(\lambda_n) = \{ \lambda_n
\}$ for all $n \geq N$. Fix $n \geq N$, and $e \in F^1
r(\lambda_n)$. Since $r(\lambda_n) = s(\lambda_{n+1}) \in
X_{n+1}$, we have $s(e) \in X_{n}$. Hence $W_1 F^* s(e)$ is
nonempty, so we may fix $\beta \in W_1 F^* s(e)$. Now $\beta e$
is the unique path in $s(\beta) F^* r(\lambda_n)$ by definition
of $X_{n+1}$. Let $\alpha$ be the unique path from $s(\beta)$
to $s(\lambda_{n_{s(\beta)}})$. Since $n_{s(\beta)} \le N \le
n$, we have $\alpha \lambda_{n_{s(\beta)}}
\lambda_{n_{s(\beta)} + 1} \ldots \lambda_{n}$ in $s(\beta) F^*
r(\lambda_n)$, and the uniqueness of this path then forces
$\beta e = \alpha \lambda_{n_{s(\beta)}} \lambda_{n_{s(\beta)}
+ 1} \ldots \lambda_{n}$, and in particular $e = \lambda_{n}$.
Thus $F^1 r(\lambda_n) = \{ \lambda_n \}$ as required.

Since $F^1 r(\lambda_n) = \{ \lambda_n \}$, we have $W_1 F^*
r(\lambda_n) = W_1 F^* \lambda_n = \{\beta\lambda_n : \beta \in
W_1 F^*s(\lambda_n)\}$. Hence that $r(\lambda_n) \in X_{n+1}$
and that $s(\lambda_n) \in X_{n}$ imply that
\[
d_{r(\lambda_n)}= \sum_{\alpha \in W_1 F^* r(\lambda_n)}d_{s(\alpha)}
= \sum_{\beta \in W_1 F^* s(\lambda_n)}d_{s(\beta)}
= d_{s(\lambda_{n})}
\]
for all $n \geq N$. This implies $d_{s(\lambda_{n})}
=d_{s(\lambda_{N})}$ for all $n \ge N$. Moreover, $\{ y \in F^0
: \text{$y \not\geq s(\lambda_n)$ for all $n$} \}$ is a
saturated hereditary subset, and the quotient of $A$ by the
ideal corresponding to this set is an AF-algebra with a
Bratteli diagram of the form
\[\begin{tikzpicture}[xscale=2]
   \foreach \x in {1,2,3,4,5} {%
       \node[inner sep=1pt] (\x) at (\x,0) {$d_{s(\lambda_{N})}$};%
   }%
   \node (6) at (6,0) {$\cdots$};%
   \foreach \x/\xx in {1/2,2/3,3/4,4/5,5/6} {%
       \draw[-latex] (\x)--(\xx);%
   }%
\end{tikzpicture}\]
Hence this quotient is isomorphic to $M_{d_{s(\lambda_{N})}}(\C)$,
which contradicts the hypothesis that $A$ has no finite-dimensional quotients.
\end{proof}

\begin{lemma} \label{f-d-quotient-char-lem}
Let $A$ be an AF-algebra.  Then $A$ has no nonzero
finite-dimensional quotients if and only if there exists a
Bratteli diagram $(E,d)$ for $A$ satisfying the following two
properties:
\begin{itemize}
\item[(1)] $d_v \geq 2$ for all $v \in E^0$; and \item[(2)]
   for all $n \in \N$ and for each $v \in V_{n+1}$ either
   $d_v > \sum_{e \in E^1v} d_{s(e)}$ or there exists $w
   \in V_n$ with $|w E^1v| \geq 2$.
\end{itemize}
\end{lemma}
\begin{proof}
If $A$ has no nonzero finite-dimensional quotients, then by
Lemma~\ref{fin-quotient-lem-1} there exists a Bratteli diagram
for $A$ satisfying~(1). Lemma~\ref{fin-quotient-lem-2} shows
that this Bratteli diagram may be telescoped to obtain a
Bratteli diagram for $A$ satisfying~(2). The vertices of the
telescoped diagram are a subset of those of the original
diagram, and the values of $d_v$ are the same for those
vertices $v$ common to both. In particular, telescoping
preserves property~(1), so the telescoped Bratteli diagram will
then satisfy both (1)~and~(2).

Conversely, suppose that there exists a Bratteli diagram
$(E,d)$ for $A$ satisfying (1) and (2). If $I$ is a proper
ideal of $A$, then $I$ corresponds to a saturated hereditary
subset $H$, and the complement $(E\setminus H, d)$ of $H$ in
$(E,d)$ is a Bratteli diagram for $A/I$. Fix a vertex $v$ in
this complement.  Since $H$ is saturated hereditary, there
exists an edge $e_1 \in E^1$ with $s(e_1) = v$ and $r(e_1)$ in the
complement also. Inductively, we may produce an infinite path
$e_1 e_2 \ldots$ in the complement.  It follows from
property~(2) that $d_{s(e_i)} < d_{s(e_{i+1})}$ for all $i$,
which implies that the function $d \colon (E\setminus H)^0 \to \N$
is unbounded. Hence $A/I$ is infinite dimensional.
\end{proof}

\begin{lemma} \label{unital-quotient-lem}
Suppose $A$ is an AF-algebra with no unital quotients.  Then
any Bratteli diagram for $A$ can be telescoped to obtain a
second Bratteli diagram $(E,d)$ for $A$ such that for all $v
\in E^0$ we have $d_v > \sum_{e \in E^1v}d_{s(e)}$.
\end{lemma}

\begin{proof}
Let $(F,d)$ be a Bratteli diagram for $A$ with $F^0$
partitioned into levels as $F^0 = \bigsqcup_{n=1}^\infty W_n$.
It suffices to show that for every $m$ there exists $n \ge m$
such that for every $v \in W_n$ we have $d_v > \sum_{\alpha
\in W_m F^* v} d_{s(\alpha)}$. Suppose not, and seek a
contradiction.
That is, we suppose that there exists $m$
such that for every $n \ge m$ the set
\[
Y_n := \Big\{ x \in W_n : d_x = \sum_{\alpha \in W_m F^*x} d_{s(\alpha)} \Big\}
\]
is nonempty.
By telescoping $(F,d)$ to $\bigsqcup_{n=m}^\infty W_n$
we may assume $m=1$.
If we let
\[
T := \{ w \in F^0 : \text{for infinitely many $n$
there exists $x \in Y_{n}$ with $w \geq x$} \},
\]
then the complement of $T$ is a saturated hereditary subset,
and the quotient of $A$ by the ideal corresponding to this
complement has a Bratteli diagram obtained by restricting to
the vertices in $T$. Along similar lines to
Lemma~\ref{fin-quotient-lem-2}, one can show that if $n \leq
p$, $x \in Y_p$, and $v \in W_n$ with $v \geq x$, then $v \in
Y_n$. Hence each $v \in T \cap W_n$ is in $Y_n$. This implies
that each $v \in T$ has the property that $d_v = \sum_{e \in
F^1v} d_{s(e)}$, and hence all the inclusions in the
corresponding directed system are unital. Thus the quotient of
$A$ considered above is unital. This contradicts the hypothesis
that $A$ has no unital quotients.
\end{proof}

\begin{lemma} \label{unital-quotient-char-lem}
Let $A$ be an AF-algebra. Then $A$ has no unital quotients if
and only if $A$ has a Bratteli diagram $(E,d)$ such that for
all $v \in E^0$ we have both $d_v \geq 2$ and $d_v > \sum_{e
\in E^1v} d_{s(e)}$.
\end{lemma}

\begin{proof}
If $A$ has no unital quotients, then the existence of such a
Bratteli diagram follows from Lemma~\ref{f-d-quotient-char-lem}
and Lemma~\ref{unital-quotient-lem}. Conversely, suppose that
$A$ has such a Bratteli diagram $(E,d)$, and fix a nonzero
quotient $A/I$ of $A$. There is a subdiagram $(F,d)$ of $(E,d)$
which is a Bratteli diagram for $A/I$. In particular $d_v >
\sum_{e \in F^1 v} d_{s(e)}$ for all $v \in F^0$. It follows
that the inclusions in the direct limit decomposition of $A$
corresponding to $(F,d)$ are all nonunital. Hence $A/I$ is
nonunital.
\end{proof}

\begin{lemma} \label{Evil-Lemma-:-(}
Let $A$ be a $C^*$-algebra which is generated by
finite-dimensional subalgebras $B$ and $C$.
Suppose that $B = \bigoplus_{v\in V} B^v$
where each $B^v \cong M_{b_v}(\C)$ and
that $C = \bigoplus_{w\in W} C^w$ where each $C^w \cong M_{c_w}(\C)$.
For each $v\in V$ suppose that $q^v$ is a minimal projection
in $B^v$ such that $q^v \in C$ and $(1_{B^v} - q^v) C = \{0\}$.
For each $v,w$,
let $m_{v,w}$ denote the rank of $q^v 1_{C^w}$ in $C^w$,
and let
\[
a_w := c_w + \sum_{v\in V}(b_v-1) m_{v,w}.
\]
Then $A = \bigoplus_{w\in W} A^w$ where each $A^w \cong
M_{a_w}(\C)$. Moreover, the inclusion $C^w \hookrightarrow A^w$
has multiplicity $1$ for $w\in W$, and the inclusion $B
\hookrightarrow A$ has multiplicity matrix
$(m_{v,w})_{v\in V,w\in W}$. Finally, the unit $1_A$ of
$A$ is equal to $\big(1_B - \sum_{v\in V}q^v\big) + 1_C$.
\end{lemma}
\begin{proof}
The assumptions on the $q^v$ imply that $\big(1_B - \sum_{v\in
V}q^v\big) + 1_C$ is the unit of $A$. To obtain the desired
decomposition of $A$, we construct a family of matrix units for
$A$. We begin by fixing convenient systems of matrix units for
the $B^v$ and the $C^w$.

For $v \in V$, let $\{\beta^v_{r,s} : 0 \le r,s \le b_v-1\}$
be a family of matrix units for $B^v$ such that
$\beta^v_{0,0} = q^v$.
Similarly, for $w \in W$
let $\{\gamma^w_{k,l} : 0 \le k,l \le c_w-1\}$
be a family of matrix units for $C^w$
such that for each $v \in V$
there exists a subset $\kappa_{v,w} \subset \{0,1,\ldots,c_w-1\}$
satisfying $q^v 1_{C^w} = \sum_{k \in \kappa_{v,w}}\gamma^w_{k,k}$.
Note that the subsets $\{\kappa_{v,w}\}_{v \in V}$
of $\{0,1,\ldots,c_w-1\}$ are mutually disjoint
and satisfy $|\kappa_{v,w}|=m_{v,w}$.

We are now ready to define the desired matrix units
for $A$;
these matrix units will be indexed by the set
\[
I_w := \big(\{0,1,\ldots,c_w-1\} \times \{0\}\big) \sqcup
\bigsqcup_{v\in V} \big(\kappa_{v,w} \times \{1,2, \ldots, b_v-1\}\big)
\]
for $w\in W$. We have $|I_w| = c_w + \sum_{v\in V}
|\kappa_{v,w}| (b_v-1) = a_w$. Define elements
$\{\alpha^w_{(k,r),(l,s)} : w \in W,\ (k,r),(l,s) \in I_w\}$ by
\[
\alpha^w_{(k,r),(l,s)} := \begin{cases}
\gamma^w_{k,l} &\text{ if $r=s=0$}\\
\gamma^w_{k,l} \beta^v_{0,s}
&\text{ if $r=0$, $l \in \kappa_{v,w}$ and $s\geq 1$} \\
\beta^{v'}_{r,0} \gamma^w_{k,l}
&\text{ if $k \in \kappa_{v',w}$, $r\geq 1$ and $s=0$} \\
\beta^{v'}_{r,0} \gamma^w_{k,l} \beta^v_{0,s}
&\text{ if $k \in \kappa_{v',w}$, $r\geq 1$,
$l \in \kappa_{v,w}$ and $s\geq 1$}.
\end{cases}
\]
We first claim that for each $w,w' \in W$, each $(k,r),(l,s) \in
I_w$ and each $(k',r'),(l',s') \in I_{w'}$,
\begin{equation}\label{eq:alpha matuits}
\alpha^w_{(k,r),(l,s)} \alpha^{w'}_{(k',r'),(l',s')} = \begin{cases}
\alpha^w_{(k,r),(l',s')} &\text{ if $w = w'$ and $(l,s) = (k',r')$}\\
0 &\text{ otherwise.}
\end{cases}
\end{equation}
To verify \eqref{eq:alpha matuits}, we consider four cases.

\noindent\textsc{Case~1:} $s=r'=0$. Since $\gamma^w_{k,l}$ are
matrix units and since the $C^w$ are orthogonal, we have
\[
\gamma^w_{k,l}\gamma^{w'}_{k',l'}=\begin{cases}
\gamma^w_{k,l'} &\text{ if $w = w'$ and $l = k'$}\\
0 &\text{ otherwise.}
\end{cases}
\]
This implies \eqref{eq:alpha matuits} in the case $s=r'=0$.

\noindent\textsc{Case~2:}
$s \ge 1$ and $r'=0$.
Then
$\beta^v_{0,s} \gamma^{w'}_{k',l'} =
\beta^v_{0,s} \beta^v_{s,s}\gamma^{w'}_{k',l'} = 0$
because $\beta^v_{s,s} \le \sum_{s=1}^{b_v-1}\beta^v_{s,s} = 1_{B^v} - q^v$
which is orthogonal to $C$ by assumption.
This shows $\alpha^w_{(k,r),(l,s)} \alpha^{w'}_{(k',r'),(l',s')} = 0$.

\noindent\textsc{Case~3:} $s = 0$ and $r' \ge 1$.
This case follows from Case~2 by taking adjoints.

\noindent\textsc{Case~4:}
$s \ge 1$ and $r' \ge 1$.
Then
\[
\gamma^w_{k,l} \beta^v_{0,s}
\beta^{v'}_{r',0} \gamma^{w'}_{k',l'}
= \begin{cases}
  \gamma^w_{k,l} \beta^v_{0,0} \gamma^{w'}_{k',l'}
    &\text{ if $v = v'$ and $s = r'$}\\
  0 &\text{ otherwise.}
\end{cases}
\]
Since $\gamma^w_{k,l}\beta^v_{0,0} =\gamma^w_{k,l}$
we have
\[
\gamma^w_{k,l} \beta^v_{0,0} \gamma^{w'}_{k',l'}
=\gamma^w_{k,l} \gamma^{w'}_{k',l'}
= \begin{cases}
  \gamma^w_{k,l'} &\text{ if $w = w'$ and $l = k'$}\\
  0 &\text{ otherwise.}
\end{cases}
\]
These show~\eqref{eq:alpha matuits} in case~4, completing the
proof of the claim.

For each $w \in W$, let $A^w:=\operatorname{span}
\{\alpha^w_{(k,r),(l,s)} : (k,r),(l,s) \in I_w\}\subset A$.
From \eqref{eq:alpha matuits},
we see that $A^w$ is isomorphic to $M_{a_w}(\C)$ for each $w \in W$,
and that $\{A^w\}_{w\in W}$ are orthogonal to each other.
We next show that $A = \sum_{w \in W}A^w$.
To see this, it suffices to show that
all the matrix units $\beta^v_{r,s}$ and $\gamma^w_{k,l}$
for $B$ and $C$ belong to $\sum_{w \in W}A^w$.
If $l \in \kappa_{v,w}$, then
\[
\gamma^w_{k,l}\beta^v_{0,0}
=(\gamma^w_{k,l}1_{C^{w}})q^v
=\gamma^w_{k,l}\Big(\sum_{l' \in \kappa_{v,w}}\gamma^{w}_{l',l'}\Big)
=\gamma^w_{k,l}.
\]
Similarly, we get
$\beta^{v'}_{0,0}\gamma^w_{k,l}=\gamma^w_{k,l}$ if $k \in
\kappa_{v',w}$. We may deduce from these two equalities that
$\alpha^w_{(k,r),(l,s)}=
\beta^{v'}_{r,0}\gamma^w_{k,l}\beta^v_{0,s}$ for all $k \in
\kappa_{v',w}$, all $r\geq 0$, all $l \in \kappa_{v,w}$ and all
$s\geq 0$. For each $v \in V$, we have
\[
\beta^v_{0,0} = q^v = \sum_{w \in W}q^v 1_{C^w}
= \sum_{w \in W}\sum_{k \in \kappa_{v,w}}\gamma^w_{k,k}.
\]
It follows that
\[
\beta^v_{r,s} = \beta^v_{r,0}\beta^v_{0,0}\beta^v_{0,s}
= \sum_{w \in W}\sum_{k \in \kappa_{v,w}}
\beta^v_{r,0}\gamma^w_{k,k}\beta^v_{0,s},
= \sum_{w \in W}\sum_{k \in \kappa_{v,w}}
\alpha^w_{(k,r),(k,s)} \in \sum_{w \in W}A^w
\]
for all $v \in V$ and all $0 \le r,s \le b_v-1$. We also have
$\gamma^w_{k,l} = \alpha^w_{(k,0),(l,0)}$ for $w \in W$ and $0
\le k,l \le c_w-1$.
Thus we get $A=\sum_{w \in W}A^w$.

It is clear that
the inclusion $C^w \hookrightarrow A^w$
has multiplicity $1$ for $w\in W$.
To see that the inclusion $B \hookrightarrow A$
has multiplicity matrix $(m_{v,w})_{v\in V,w\in W}$,
it suffices to see that for each $v \in V$ and $w \in W$,
the product of the minimal projection $q^v \in B^v$
and the unit $1_{A^w}$ of $A^w$
has rank $m_{v,w}$ in $A^w \cong M_{a_w}(\C)$.
Since $q^v \in C$, we have
\[
q^v1_{A^w}=q^v1_{C^w}
= \sum_{k \in \kappa_{v,w}}\gamma^w_{k,k}
= \sum_{k \in \kappa_{v,w}}\alpha^w_{(k,0),(k,0)}.
\]
This shows that the rank of $q^v1_{A^w} \in A^w$
is $|\kappa_{v,w}| = m_{v,w}$.
\end{proof}

\section{Realizations of AF-algebras} \label{results-sec}

\subsection{A construction of an ultragraph from a certain type of Bratteli diagram}  \label{ultra-contruct-subsec}

In this section we show how to construct ultragraphs from
certain Bratteli diagrams and use these ultragraphs to realize
particular classes of AF-algebras as ultragraph $C^*$-algebras,
Exel-Laca algebras, and graph $C^*$-algebras.

\begin{definition} \label{ultragraph-def}
Let $A$ be an AF-algebra with no nonzero finite-dimensional
quotients.  By Lemma~\ref{f-d-quotient-char-lem} there exists a
Bratteli diagram $(E,d)$ for $A$ satisfying the following two
properties:
\begin{itemize}
\item[(1)] $d_v \geq 2$ for all $v \in E^0$; and
\item[(2)] for all $n \in \N$ and for each $v \in V_{n+1}$
   either $d_v > \sum_{e \in E^1 v} d_{s(e)}$ or there
   exists $w \in V_n$ with $|w E^1v| \geq 2$.
\end{itemize}
We define
\[
\Delta_v := d_v - \sum_{e \in E^1v} (d_{s(e)} - 1).
\]
The symbol $\Delta$ has been chosen to connote ``difference".
Note that from the property (1), $\Delta_v = d_v$ if and only
if $v$ is a source.  In addition, it follows from the
properties of our Bratteli diagram that $\Delta_v \geq 2$ for
all $v \in E^0$.

We claim that for each $v \in E^0$ there exists an injection
$k_v\colon E^1 v \to \{0,1,\ldots, \Delta_v-1\}$ such that
there exists $e \in E^1 v$ with $k_v(e) = 0$ if and only if
$d_v = \sum_{e \in E^1 v}d_{s(e)}$, and in this case $e$ is not
the only element of $s(e)E^1v$. To justify this claim, first
observe that
\begin{align*}
\Delta_v &= d_v - \sum_{e \in E^1 v} (d_{s(e)} - 1) = d_v - \sum_{e \in E^1 v} d_{s(e)} + \sum_{e \in E^1 v} 1 = \big( d_v - \sum_{e \in E^1 v} d_{s(e)} \big) + |E^1v|
\end{align*}
Hence if $d_v  >  \sum_{e \in E^1 v} d_{s(e)}$ we may always
choose an injection $k_v\colon E^1v \to \{0,1,\ldots,
\Delta_v-1\}$ so that its image does not contain $0$. On the
other hand if $d_v = \sum_{e \in E^1 v} d_{s(e)}$, then by
hypothesis on the Bratteli diagram there exists $w \in E^0$
with $|wE^1v| \geq 2$ so we may choose a bijection $k_v\colon
E^1v \to \{0,1,\ldots, \Delta_v-1\}$ such that $e \in E^1v$
with $k_v(e) = 0$ satisfies $s(e)=w$. This establishes the
claim.

We now define an ultragraph $\G = (G^0, \G^1, r_\G, s_\G)$ by
\[
G^0 := \{ v_i : v \in E^0 \text{ and } 1 \leq i \leq \Delta_v -1 \}
\qquad \text{ and } \qquad \G^1 := \{ e_{v_i} : v_i \in G^0
\}
\]
with
\[
s_\G (e_{v_i}) := v_i \quad \text{ for all $v_i \in G^0$},
\qquad  \qquad r_\G (e_{v_i}) := \{v_{i-1}\} \quad
\text{ for $2 \leq i \leq \Delta_v-1$}
\]
and
\begin{align*}
r_\G(e_{v_1}) := \big\{ w_k :
\text{ there exists a path $\lambda=\lambda_1\lambda_2\ldots \lambda_n$
such that $s(\lambda)=v$, $r(\lambda)=w$,}& \\
\text{$k_{r(\lambda_i)}(\lambda_i)=0$ for $i=1,2,\ldots, n-1$, and $k_{w}(\lambda_n)=k\ge 1$}&\big\}.
\end{align*}
\end{definition}

To check that $\G$ is an ultragraph, we only need to see that
$r_\G(e_{v_1}) \neq \emptyset$.

\begin{lemma} \label{lem:r_G(e_{v_1})}
For all $n$ and $v \in V_n$,
the set $r_\G(e_{v_1})$ is nonempty and satisfies
\[
r_\G(e_{v_1})
= \{ w_{k_w(e)} : w \in V_{n+1}, e \in vE^1 w, k_w(e)\ge 1\}
\cup \bigcup_{w \in V_{n+1}, e \in vE^1 w, k_w(e)=0} r_\G(e_{w_1}).
\]
\end{lemma}

\begin{proof}
The latter equality follows from the definition of $r_\G(e_{v_1})$.
For each $v \in V_n$,
there exists $w \in V_{n+1}$
such that $vE^1 w \neq \emptyset$.
By the assumption on $k_w$,
there exists $e \in vE^1 w$ such that $k_w(e) \geq 1$.
Thus $w_{k_w(e)} \in r_\G(e_{v_1})$.
This shows that $r_\G(e_{v_1})$ is nonempty.
\end{proof}

\begin{remark}
By definition,
$r_\G(e_{v_1}) \subset \bigcup_{k=n+1}^\infty V_k$ for $v \in V_n$.
One can show that
this property together with the equality in Lemma~\ref{lem:r_G(e_{v_1})}
uniquely determines $\{r_\G(e_{v_1})\}_{v \in E^0}$.
\end{remark}


\begin{example} \textbf{An example of the ultragraph construction:} Consider a Bratteli diagram $(E,d)$ satisfying Conditions
(1)~and~(2) of Lemma~\ref{f-d-quotient-char-lem} and whose
first three levels are as illustrated below. In the diagram,
each vertex is labeled with its name, and above the label $a$
appears the integer $d_a$.
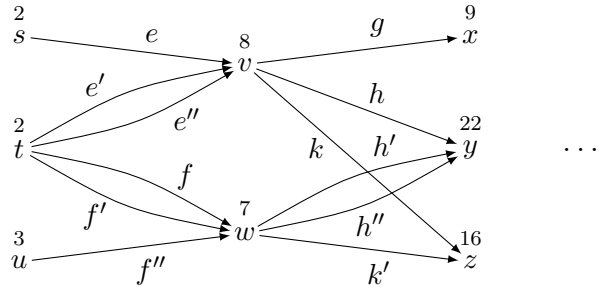
\[\begin{tikzpicture}[scale=1.5]
  \node[circle,inner sep=0.5pt] (s) at (0,1) {\small$s$};%
  \node[inner sep=1.5pt,anchor=south] at (s.north) {\tiny2};%
  \node[circle,inner sep=0.5pt] (t) at (0,0) {\small$t$};%
  \node[inner sep=1.5pt,anchor=south] at (t.north) {\tiny2};%
  \node[circle,inner sep=0.5pt] (u) at (0,-1) {\small$u$};%
  \node[inner sep=1.5pt,anchor=south] at (u.north) {\tiny3};%
  \node[circle,inner sep=0.5pt] (v) at (2,0.75) {\small$v$};%
  \node[inner sep=1.5pt,anchor=south] at (v.north) {\tiny8};%
  \node[circle,inner sep=0.5pt] (w) at (2,-0.75) {\small$w$};%
  \node[inner sep=1.5pt,anchor=south] at (w.north) {\tiny7};%
  \node[circle,inner sep=0.5pt] (x) at (4,1) {\small$x$};%
  \node[inner sep=1.5pt,anchor=south] at (x.north) {\tiny9};%
  \node[circle,inner sep=0.5pt] (y) at (4,0) {\small$y$};%
  \node[inner sep=1.5pt,anchor=south] at (y.north) {\tiny22};%
  \node[circle,inner sep=0.5pt] (z) at (4,-1) {\small$z$};%
  \node[inner sep=1.5pt,anchor=south] at (z.north) {\tiny16};%
  \node at (5,0) {\dots};%
%
  \draw[-latex] (s)-- node[above,pos=0.6] {\small$e$} (v) node[pos=0.5,anchor=south,inner sep=1pt]  {};
  \draw[opacity=0,color=white] (t)--(v) node[pos=0.5,sloped,inner sep=0.25cm] (tv) {};
  \draw[-latex] (t) .. controls (tv.north) .. node[above,pos=0.3] {\small$e'$} (v);
  \draw[-latex] (t) .. controls (tv.south) .. node[below,pos=0.8] {\small$e''$} (v);
  \draw[opacity=0,color=white] (t)--(w) node[pos=0.5,sloped,inner sep=0.25cm] (tw) {};
  \draw[-latex] (t) .. controls (tw.north) .. node[above,pos=0.8] {\small$f$}(w);
  \draw[-latex] (t) .. controls (tw.south) .. node[below,pos=0.3] {\small$f'$}(w);

  \draw[-latex] (u)-- node[below,pos=0.6] {\small$f''$} (w) node[pos=0.5,anchor=south,inner sep=1pt] {};
  \draw[-latex] (v)-- node[above,pos=0.6] {\small$g$} (x) node[pos=0.5,anchor=south,inner sep=1pt] {};
  \draw[-latex] (v)-- node[above,pos=0.6] {\small$h$} (y) node[pos=0.5,anchor=south,inner sep=1pt] {};
  \draw[-latex] (v)-- node[below,pos=0.3] {\small$k$} (z) node[pos=0.5,anchor=south,inner sep=1pt] {};
  \draw[opacity=0,color=white] (w)--(y) node[pos=0.5,sloped,inner sep=0.25cm] (vy) {};
  \draw[-latex] (w) .. controls (vy.north) .. node[above,pos=0.7] {\small$h'$}(y);
  \draw[-latex] (w) .. controls (vy.south) .. node[below,pos=0.55] {\small$h''$} (y);
  \draw[-latex] (w)-- node[below,pos=0.6] {\small$k'$} (z) node[pos=0.5,anchor=south,inner sep=1pt] {};
\end{tikzpicture}\]
The values of $\Delta$ for the vertices visible in the diagram
are
\[\begin{array}{c@{\extracolsep{2em}}c@{\extracolsep{2em}}c}
\Delta_s = 2 & \Delta_v = 5 &  \Delta_x = 2 \\
\Delta_t = 2 & \Delta_w = 3 & \Delta_y= 3\\
\Delta_u = 3 & \ & \Delta_z = 3 \\
\end{array}\]
So the corresponding section of the resulting ultragraph $\G$
will have vertices
\[
G^0 = \{s_1, t_1, u_1, u_2, v_1, v_2, v_3, v_4, w_1, w_2,
       x_1, y_1, y_2, z_1, z_2, \ldots \},
\]
and each of these vertices ${a_i}$ will emit exactly one
ultraedge $e_{a_i}$. For $i \not=1$, we have $r_\G(e_{a_i}) =
\{a_{i-1}\}$. To determine the ranges of the $e_{a_1}$, we must
choose injections $k_a\colon E^1a \to \{0,1,\ldots,
\Delta_a-1\}$ for $a \in E^0$ with the properties described
above; in particular, this necessitates that $0$ is in the
image of $k_a$ only when $a=w$ or $a=y$, and also that
$k_w(f'')\neq 0$ and $k_y(h)\neq 0$.

One possible set of choices of injections $k_a$ is
\begin{align*}
&\begin{array}[t]{ccc}
k_v(e)=1,& k_v(e')=3,& k_v(e'')=4,\\
k_w(f)=0,& k_w(f')=2,& k_w(f'')=1,
\end{array}
&
&\begin{array}[t]{ccc}
k_x(g)=1,&& \\
k_y(h)=1,& k_y(h')=0,& k_y(h'')=2,\\
k_z(k)=2,& k_z(k')=1.&
\end{array}
\end{align*}
We can calculate
\begin{gather*}
r_\G(e_{s_1}) = \{v_1\},
   \qquad  r_\G(e_{t_1}) = \{v_3, v_4, w_2, y_2, z_1\} \cup r_\G(e_{y_1}),
   \qquad r_\G(e_{u_1}) = \{w_1\},\\
r(e_{v_1}) = \{x_1, y_1, z_2\},
   \quad\text{ and }\quad
   r(e_{w_1}) = \{z_1, y_2\} \cup r_\G(e_{y_1}).
\end{gather*}
We may now draw the fragment of the ultragraph $\G$
corresponding to the given fragment of the Bratteli diagram
$(E,d)$.
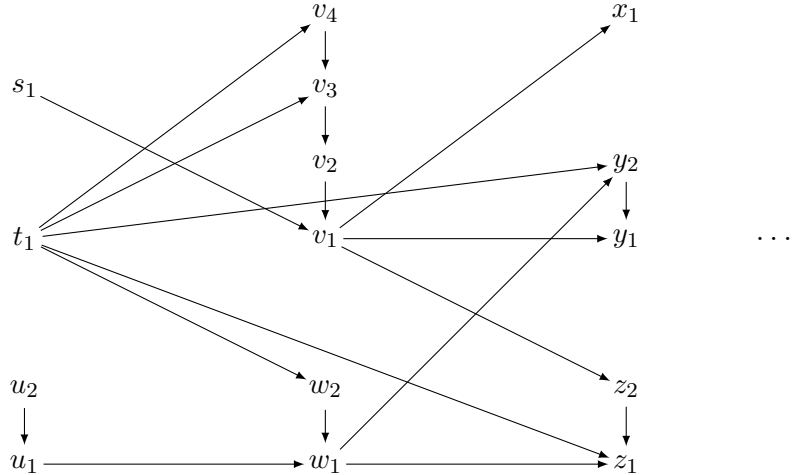
\[\begin{tikzpicture}[scale=2]
  \node[circle,inner sep=0.5pt] (s1) at (0,-1) {\small$u_1$};%
  \node[circle,inner sep=0.5pt] (s2) at (0,-0.5) {\small$u_2$};%
  \node[circle,inner sep=0.5pt] (t1) at (0,0.5) {\small$t_1$};%
  \node[circle,inner sep=0.5pt] (u1) at (0,1.5) {\small$s_1$};%
  \node[circle,inner sep=0.5pt] (v1) at (2,-1) {\small$w_1$};%
  \node[circle,inner sep=0.5pt] (v2) at (2,-0.5) {\small$w_2$};%
  \node[circle,inner sep=0.5pt] (w1) at (2,0.5) {\small$v_1$};%
  \node[circle,inner sep=0.5pt] (w2) at (2,1) {\small$v_2$};%
  \node[circle,inner sep=0.5pt] (w3) at (2,1.5) {\small$v_3$};%
  \node[circle,inner sep=0.5pt] (w4) at (2,2) {\small$v_4$};%
  \node[circle,inner sep=0.5pt] (x1) at (4,-1) {\small$z_1$};%
  \node[circle,inner sep=0.5pt] (x2) at (4,-0.5) {\small$z_2$};%
  \node[circle,inner sep=0.5pt] (y1) at (4,0.5) {\small$y_1$};%
  \node[circle,inner sep=0.5pt] (y2) at (4,1) {\small$y_2$};%
  \node[circle,inner sep=0.5pt] (z1) at (4,2) {\small$x_1$};%
  \node at (5,0.5) {\dots};%
%
  \draw[-latex] (s2)--(s1);
  \draw[-latex] (s1)--(v1);
  \draw[-latex] (t1)--(v2);
  \draw[-latex] (t1)--(w3);
  \draw[-latex] (t1)--(w4);
  \draw[-latex] (t1)--(x1);
  \draw[-latex] (t1)--(y2);
  \draw[-latex] (u1)--(w1);
  \draw[-latex] (v2)--(v1);
  \draw[-latex] (v1)--(x1);
  \draw[-latex] (v1)--(y2);
  \draw[-latex] (w4)--(w3);
  \draw[-latex] (w3)--(w2);
  \draw[-latex] (w2)--(w1);
  \draw[-latex] (w1)--(x2);
  \draw[-latex] (w1)--(y1);
  \draw[-latex] (w1)--(z1);
  \draw[-latex] (x2)--(x1);
  \draw[-latex] (y2)--(y1);
\end{tikzpicture}\]
Note that by definition of the ultragraph $\G$, each vertex
emits exactly one ultraedge, so in the picture any multiple
arrows leaving the same vertex actually have the same label and
constitute a single ultraedge of $\G$.
\end{example}

\subsection{Sufficient conditions for realizations} \label{sufficient-cond-subsec}

\begin{theorem} \label{AF-as-ultra-thm}
Let $A$ be an AF-algebra with a Bratteli diagram satisfying the
conditions of Lemma~\ref{f-d-quotient-char-lem}.   If $\G$ is
an ultragraph constructed from this Bratteli diagram as in
Definition~\ref{ultragraph-def}, then $A \cong C^*(\G)$. In
addition, $C^*(\G)$ is an Exel-Laca algebra.
\end{theorem}

\begin{proof}
Let $(E,d)$ be a Bratteli diagram for $A$ with the vertices
partitioned into levels as $E^0 = \bigsqcup_{n=1}^\infty V_n$
and satisfying the conditions of
Lemma~\ref{f-d-quotient-char-lem}, and let $\G$ be an
ultragraph constructed from $(E,d)$ as in
Definition~\ref{ultragraph-def}. Our strategy is to find a
direct limit decomposition of $C^*(\G)$ so that at each level
we may apply Lemma~\ref{Evil-Lemma-:-(} to see that the
inclusion of finite-dimensional algebras is the same as the
corresponding inclusion in the direct limit decomposition of
$A$ determined by $(E,d)$.

For each $v \in E^0$ let
\[
C^v := C^* ( \{ s_{e_{v_i}} : 1 \leq i \leq \Delta_v - 1 \}).
\]
We have $s_{e_{v_i}}s_{e_{v_i}}^*=p_{v_i}$
for $1 \leq i \leq \Delta_v - 1$ and
$s_{e_{v_i}}^*s_{e_{v_i}}=p_{v_{i-1}}$
for $2 \leq i \leq \Delta_v - 1$.
We define a projection
$q^v := p_{r_\G(e_{v_1})}=s_{e_{v_1}}^*s_{e_{v_1}} \in C^v$,
which is orthogonal to $p_{v_i}$ for $1 \leq i \leq \Delta_v - 1$.
These computations show that there exist matrix units
$\{\gamma^v_{k,l} : 0\leq k,l \leq \Delta_v -1\}$ in $C^v$
such that $\gamma^v_{0,0}=q^v$, $\gamma^v_{i,i}=p_{v_i}$
and $\gamma^v_{i,i-1}=s_{e_{v_i}}$
for $1 \leq i \leq \Delta_v - 1$.
Explicitly, $\gamma^v_{k,l} \in C^v$ is given by
\[
\gamma^v_{k,l} :=
s_{e_{v_k}}s_{e_{v_{k-1}}} \cdots s_{e_{v_1}}
q^vs_{e_{v_1}}^*s_{e_{v_2}}^* \cdots s_{e_{v_l}}^*
\]
for $0\leq k,l \leq \Delta_v -1$.
This shows that $C^v$ is isomorphic to $M_{\Delta_v} (\C)$
with minimal projection $q^v$
and the unit $\sum_{i=1}^{\Delta_v - 1}p_{v_{i}}+q^v$.
For each $n \in \N$
\[
C_n := C^*( \{ s_{e_{v_i}} : v \in V_n \text{ and } 1
\leq i \leq \Delta_v - 1 \} )
\]
is equal to $\bigoplus_{v \in V_n} C^v$.
Moreover, for $n \in \N$, define
\[\textstyle
B_n := C^*\big(\bigcup_{j=1}^n C_j \big)
= C^*\big( \{ s_{e_{v_i}} : v \in \bigcup_{j=1}^n V_j \text{ and } 1 \leq i \leq \Delta_v-1 \} \big).
\]
\textbf{Claim:} For each $n \in \N$,
the unit $1_{B_n}$ of $B_n$ is given by
$\sum_{v \in \bigcup_{j=1}^n V_j}\sum_{i=1}^{\Delta_v -1}p_{v_i}
+ \sum_{v\in V_n}q^v$, and
there exists a decomposition $B_n = \bigoplus_{v \in V_n} B^v$
such that each $B^v \cong M_{d_v} (\C)$ with
minimal projection $q^v$;
and for each $n \in \N$, the inclusion $B_n \hookrightarrow B_{n+1}$
has multiplicity matrix $( |vE^1w| )_{v \in V_n, w \in V_{n+1}}$.

We proceed by induction on $n$.
When $n =1$, let $B^v := C^v$ for $v \in V_1$.
Then $B_1= C_1$ has the decomposition
$B_1=\bigoplus_{v \in V_1} B^v$.
For each $v \in V_1$,
we have $\Delta_v = d_v$ because $v$ is a source.
Hence $B^v = C^v$ is isomorphic to $M_{d_v} (\C)$
with minimal projection $q^v$
and the unit $\sum_{i=1}^{\Delta_v - 1}p_{v_{i}}+q^v$.
This shows the claim in the case $n=1$.
For the inductive step, assume that $B_n$ has the desired decomposition.
To apply Lemma~\ref{Evil-Lemma-:-(}
to the \Ca $B_{n+1}$ which is generated by $B_n$ and $C_{n+1}$,
we check that for each $v \in V_n$
the minimal projection $q^{v} \in B^v$ is in $C_{n+1}$ and
satisfies $(1_{B^v} - q^{v}) C_{n+1} = \{0\}$.
We see that
\[
\sum_{v \in V_{n}}(1_{B^v} - q^{v})
=1_{B_n} - \sum_{v \in V_{n}}q^{v}
=\sum_{v \in \bigcup_{j=1}^n V_j}\sum_{i=1}^{\Delta_v -1}p_{v_i}
\]
which is orthogonal to $C_{n+1}$.
This proves $(1_{B^v} - q^{v}) C_{n+1} = \{0\}$
for all $v \in V_{n}$.
For each $v \in V_{n}$,
Lemma~\ref{lem:r_G(e_{v_1})} implies
\begin{equation}\label{eq:qv=}
q^{v}=p_{r_\G(e_{v_1})}
  = \sum_{w \in V_{n+1}}\Big(
      \sum_{\substack{e \in vE^1 w\\ k_w(e)\ge 1}} p_{w_{k_w(e)}} + \sum_{\substack{e \in vE^1 w \\ k_w(e)=0}} p_{r_G(e_{w_1})}\Big)
= \sum_{w \in V_{n+1}}\sum_{e \in vE^1 w}\gamma^w_{k_w(e),k_w(e)}.
\end{equation}
Hence $q^{v} \in C_{n+1}$. Thus we can apply
Lemma~\ref{Evil-Lemma-:-(} to obtain the decomposition $B_{n+1}
= \bigoplus_{w \in V_{n+1}} B^w$. Since the inclusion $C^w
\hookrightarrow B^w$ has multiplicity $1$ for $w\in W$, the
projection $q^{w}$ is minimal in $B^w$. From \eqref{eq:qv=},
$q^{v} 1_{C^w}$ has rank $|vE^1w|$ in $C^w$ for $w \in
V_{n+1}$. The definition of $\Delta_w$ implies that
\[
d_w = \Delta_w + \sum_{w \in V_{n+1}} (d_v-1) |vE^1w|.
\]
Hence $B^w$ is isomorphic to $M_{d_w} (\C)$ for $w \in V_{n+1}$.
The conclusion of Lemma~\ref{Evil-Lemma-:-(}
also shows that the inclusion $B_n \hookrightarrow B_{n+1}$
has multiplicity matrix $( |vE^1w| )_{v \in V_n, w \in V_{n+1}}$,
and that the unit of $B_{n+1}$ is equal to
$\sum_{v \in \bigcup_{j=1}^{n+1} V_j}\sum_{i=1}^{\Delta_v -1}p_{v_i}
+ \sum_{w\in V_{n+1}}q^w$.
This proves the claim.

We see that $\bigcup_{n=1}^\infty B^n$ contains $\{ s_e : e \in \G^1 \}$.
Since each vertex $v$ in $\G$ emits exactly one ultraedge $e$,
$p_v = s_es_e^*$ is contained in $\bigcup_{n=1}^\infty B^n$.
Thus $\bigcup_{n=1}^\infty B^n$ contains
all the generators of $C^*(\G)$.  Hence $C^*(\G) =
\overline{\bigcup_{n=1}^\infty B^n}$ is an AF-algebra, and the
preceding paragraphs show that $(E,d)$ is a Bratteli diagram
for $C^*(\G)$, giving $A \cong C^*(\G)$.
Since every vertex of $\G$ emits exactly one ultraedge,
$C^*(\G)$ is an Exel-Laca algebra (see Remark~\ref{rem:ELalgis,,,}).
\end{proof}

\begin{corollary} \label{no-nonzero-f-d-quotient-then-EL}
If $A$ is an AF-algebra with no nonzero finite-dimensional
quotients, then $A$ is isomorphic to an Exel-Laca algebra.
\end{corollary}

\begin{proof}
Since $A$ has no nonzero finite-dimensional quotients,
Lemma~\ref{f-d-quotient-char-lem} implies that $A$ has a
Bratteli diagram satisfying the conditions stated.  It follows
from Theorem~\ref{AF-as-ultra-thm} that $A$ is isomorphic to an
Exel-Laca algebra.
\end{proof}

The following result is important in that it is one of the few
instances where we can give a complete characterization of
AF-algebras in a certain graph $C^*$-algebra class.
In particular, we give necessary and sufficient conditions
for an AF-algebra to be the $C^*$-algebra
of a row-finite graph with no sinks.

\begin{theorem} \label{no-unital-quotient-then-graph-alg}
Let $A$ be an AF-algebra.  Then the following are equivalent:
\begin{enumerate}
\item $A$ has no (nonzero) unital quotients.
\item $A$ is isomorphic to the $C^*$-algebra of a
  row-finite graph with no sinks.
\end{enumerate}
\end{theorem}

\begin{proof}
We shall first prove that $(1)$ implies $(2)$. Suppose that $A$
has no unital quotients. By
Corollary~\ref{unital-quotient-char-lem} there is a Bratteli
diagram $(E,d)$ for $A$ such that for all $v \in E^0$ we have
both $d_v \geq 2$ and $d_v > \sum_{e \in E^1 v} d_{s(e)}$.  Let
$\G$ be an ultragraph constructed from $(E,d)$ as in
Definition~\ref{ultragraph-def}. Theorem~\ref{AF-as-ultra-thm}
implies that $A \cong C^*(\G)$. Furthermore, since $d_v >
\sum_{e \in E^1 v} d_{s(e)}$, we have $k_v(e) \geq 1$ for all
$v \in E^0$ and $e \in E^1v$. For $v \in E^0$,
Lemma~\ref{lem:r_G(e_{v_1})} implies $r_\G(e_{v_1}) = \{
w_{k_w(e)} : w \in V_{n+1}, e \in vE^1 w, k_w(e)\ge 1\}$. Thus,
$r_\G(e)$ is finite for every $e \in \G^1$. Hence $C^*(\G)$ is
isomorphic to a graph $C^*$-algebra of a row-finite graph with
no sinks (see \cite[Remark~5.25]{KMST2}).

We next prove that $(2)$ implies $(1)$.
Suppose that $A \cong C^*(E)$, where $E$ is a row-finite graph with no sinks.
Since $C^*(E)$ is an AF-algebra, it follows from
\cite[Theorem~2.4]{KPR} that $E$ has no cycles.  Thus $E$
satisfies Condition~(K), and \cite[Theorem~4.4]{BPRS} implies
that every ideal of $C^*(E)$ is gauge invariant.  Suppose $I$
is a proper ideal of $C^*(E)$.  Then $I = I_H$ for some
saturated hereditary proper subset $H \subset E^0$, and $C^*(E)
/ I_H \cong C^*(E_H)$, where $E_H$ is the nonempty subgraph of
$E$ with $E_H^0 := E^0 \setminus H$ and $E_H^1 := \{ e \in E^1
: r(e) \notin H \}$ (see \cite[Theorem~4.1]{BPRS}).  Since $H$
is saturated hereditary, that $E$ has no sinks implies that
$E_H$ has no sinks. Since $E$ has no cycles, $E_H$ also has no
cycles.  Because $E_H$ is a nonempty graph with no cycles and
no sinks, $E_H^0$ is infinite. Thus $C^*(E_H)$ is nonunital
\cite[Proposition~1.4]{KPR}.
\end{proof}

\begin{corollary} \label{cor:stable}
Let $A$ be a stable AF-algebra. Then there is a row-finite
graph $E$ with no sinks such that $A \cong C^*(E)$. In
particular, $A$ is isomorphic to a graph $C^*$-algebra, to an
Exel-Laca algebra, and to an ultragraph $C^*$-algebra.
\end{corollary}
\begin{proof}
Since any nonzero quotient of a stable $C^*$-algebra is stable,
every quotient of $A$ is stable, and in particular nonunital.
The result then follows from
Theorem~\ref{no-unital-quotient-then-graph-alg}.
\end{proof}

\begin{lemma}\label{lem:M_2(C^*(G))}
Let $\G = (G^0, \G^1, r,s)$ be an ultragraph. Let $\wt{\G} =
(\wt{G}^0, \wt{\G}^1, \tilde{r},\tilde{s})$ be the ultragraph
defined by $\wt{G}^0:=G^0\sqcup\{v_0\}$ and
$\wt{\G}^1:=\G^1\sqcup\{e_0\}$ with
\[
\tilde{s}|_{\G^1}=s,\qquad \tilde{s}(e_0)=v_0,\qquad
\tilde{r}|_{\G^1}=r, \qquad \text{ and } \qquad \tilde{r}(e_0)=G^0.
\]
Then $C^*\big(\wt{\G}\big)\cong M_2(\MU{C^*(\G)})$, where
$\MU{C^*(\G)}$ is the minimal unitization of $C^*(\G)$.
\end{lemma}
\begin{proof}
We first notice that the algebra $\wt{\G}^0$
is generated by the algebra $\G^0 \subseteq \Pow(\wt{G}^0)$
and the two elements $G_0, \{v_0\} \in \Pow(\wt{G}^0)$.
The universal property of $C^*(\wt{\G})$ implies that
there is a \shom
$\phi\colon C^*(\wt{\G}) \to M_2(\MU{C^*(\G)})$ satisfying
\[
\phi(p_A) = \left( \begin{smallmatrix} p_A & 0 \\ 0 & 0 \end{smallmatrix} \right) \text{ for all $A \in \G^0$}
\qquad \text{ and } \qquad
\phi(s_e) = \left( \begin{smallmatrix} s_e & 0 \\ 0 & 0 \end{smallmatrix} \right) \text{ for all $e \in \G^1$}
\]
and
\[
\phi(p_{G^0}) = \left( \begin{smallmatrix} 1 & 0 \\ 0 & 0 \end{smallmatrix} \right), \quad
\phi(p_{v_0}) = \left( \begin{smallmatrix} 0 & 0 \\ 0 & 1 \end{smallmatrix} \right), \  \text{ and } \
\phi(s_{e_0}) = \left( \begin{smallmatrix} 0 & 0 \\ 1 & 0 \end{smallmatrix} \right).
\]
The Gauge-Invariant Uniqueness Theorem \cite[Theorem~6.8]{Tom}
shows that $\phi$ is injective. Standard calculations show that
the image under $\phi$ of the generating Cuntz-Krieger
$\widetilde{\G}$-family in $C^*(\widetilde{\G})$
generates $M_2(\MU{C^*(\G)})$.
Hence $\phi$ is an isomorphism.
\end{proof}

\begin{corollary} \label{A-AF-M_2-AF}
Let $A$ be a $C^*$-algebra,
and let $\MU{A}$ denote the minimal unitization of $A$.
If $A$ is isomorphic to an Exel-Laca algebra,
then $M_2(\MU{A})$ is isomorphic to an Exel-Laca algebra.
\end{corollary}

\begin{proof}
If $A$ is isomorphic to an Exel-Laca algebra,
then by Remark~\ref{rem:ELalgis,,,}
$A \cong C^*(\G)$ where $\G$ is an ultragraph
with bijective source map.
By Lemma~\ref{lem:M_2(C^*(G))}
$C^*(\wt{\G}) \cong M_2 (\MU{A})$, and since $\wt{\G}$ is
an ultragraph with bijective source map,
$C^*(\wt{\G})$ is an Exel-Laca algebra.
\end{proof}

The following example shows that the converse of
Corollary~\ref{no-nonzero-f-d-quotient-then-EL} does not hold.

\begin{example} \label{E-L-quotient-not-Ex}
Let $A$ be a nonunital, simple AF-algebra (such as $\K$).  By
Corollary~\ref{simple-AF-graph-EL} $A$ is isomorphic to an
Exel-Laca algebra, and by Corollary~\ref{A-AF-M_2-AF}
$M_2(\MU{A})$ is an Exel-Laca algebra. However, $M_2(\MU{A})$
has a quotient isomorphic to the finite-dimensional
$C^*$-algebra $M_2(\C)$.  Thus the converse of
Corollary~\ref{no-nonzero-f-d-quotient-then-EL} does not hold.
(It is also worth mentioning that $M_2(\C)$ is a quotient of an
Exel-Laca algebra, but $M_2(\C)$ is not itself an Exel-Laca
algebra; cf.~Corollary~\ref{No-fin-dim-E-L-Cor}.)
\end{example}

The following elementary example shows that the $C^*$-algebra
of a row-finite graph with sinks may admit unital quotients
(cf.~Theorem~\ref{no-unital-quotient-then-graph-alg}).

\begin{example}
The AF-algebra $M_2(\C) \oplus M_2(\C)$ is isomorphic to the
$C^*$-algebra of the graph $\bullet \longleftarrow \bullet
\longrightarrow \bullet$ by \cite[Corollary~2.3]{KPR}. However,
this $C^*$-algebra has $M_2(\C)$ as a unital quotient. Thus
graphs with sinks can have associated $C^*$-algebras that are
AF-algebras with proper unital quotients.
\end{example}

The next example is more intriguing. Before considering this
example, one is tempted to believe that if $E$ is a row-finite
graph, then $C^*(E)$ is isomorphic to a direct sum of a
countable collection of algebras of compact operators on
(finite or countably infinite dimensional) Hilbert spaces and
the $C^*$-algebra of a row-finite graph with no sinks (see
Proposition~\ref{prp:direct-sums-as-graph-algs}). This would
give a characterization of AF-algebras associated to row-finite
graphs along similar lines to
Theorem~\ref{no-unital-quotient-then-graph-alg}. However, the
example shows that this is not the case in general.

\begin{example}
Let $E$ be the graph
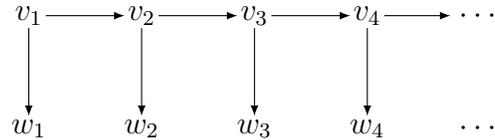
\[\begin{tikzpicture}[scale=1.5]
   \foreach \x in {1,2,3,4} {%
       \node[inner sep=1pt] (v\x) at (\x,1) {$v_{\x}$};%
       \node[inner sep=1pt] (w\x) at (\x,0) {$w_{\x}$};%
   }%
   \node[inner sep=3pt] (v5) at (5,1) {$\cdots$};%
   \node[inner sep=3pt] (w5) at (5,0) {$\cdots$};%
   \foreach \x/\xx in {1/2,2/3,3/4,4/5} {%
       \draw[-latex] (v\x)--(v\xx);%
       \draw[-latex] (v\x)--(w\x);%
   }%
\end{tikzpicture}\]
Then for each $n \in \N$ the set $H_n := \{ v_n, v_{n+1},
\ldots \} \cup \{ w_n, w_{n+1}, \ldots \}$ is a saturated
hereditary subset of $E$, and $C^*(E) / I_{H_n}$ is a
finite-dimensional $C^*$-algebra.  Thus $C^*(E)$ is an
AF-algebra with infinitely many finite-dimensional quotients.
This shows that, unlike what occurs for row-finite graphs with
no sinks (cf.~Theorem~\ref{no-unital-quotient-then-graph-alg}),
the situation with sinks is much more complicated.  It also
shows that $C^*(E)$ does not have a Bratteli diagram of the
types described in Lemma~\ref{unital-quotient-lem} or
Lemma~\ref{unital-quotient-char-lem}.  Hence our construction
of the ultragraph described in \S\ref{ultra-contruct-subsec}
cannot be applied.
\end{example}

By eliminating the bad behavior arising in the preceding
example, we obtain a limited extension of
Theorem~\ref{no-unital-quotient-then-graph-alg} to graphs
containing sinks.

\begin{proposition}\label{prp:direct-sums-as-graph-algs}
Let $A$ be an AF algebra. Then the following are equivalent:
\begin{enumerate}
\item\label{it:dir-sums:only-if} $A$ is isomorphic to the
  $C^*$-algebra of a row-finite graph in which each
  vertex connects to at most finitely many sinks; and
\item\label{it:dir-sums:if} $A$ has the form
  $\big(\bigoplus_{x \in X} M_{n_x}(\C)\big) \oplus A'$
  where $X$ is an at most countably-infinite index set,
  each $n_x$ is a positive integer, and $A'$ is an AF algebra
  with no unital quotients.
\end{enumerate}
\end{proposition}
\begin{proof}
To see that
(\ref{it:dir-sums:only-if})~implies~(\ref{it:dir-sums:if}), we
let $E$ be a row-finite graph in which each vertex connects to
at most finitely many sinks and such that $A \cong C^*(E)$.
Since $A$ is an AF-algebra, $E$ has no cycles. Let $\sinks(E)$
denote the collection $\{v \in E^0 : vE^1 = \emptyset\}$ of
sinks in $E$. Let $H$ be the smallest saturated hereditary
subset of $E^0$ containing $\sinks(E)$. Since each vertex
connects to at most finitely many sinks, $H$ is equal to the
set of $v \in E^0$ such that $vE^n = \emptyset$ for some $n$.
Let $F$ be the graph with vertices $F^0 := E^0 \setminus H$,
edges $F^1 = \{e \in E^1 : r(e) \not\in H\}$ and range and
source maps inherited from $E$. Note that the description of
$H$ above implies that $F$ has no sinks; moreover $F$ is
row-finite because $E$ is. We claim that
\[\textstyle
C^*(E) \cong \big(\bigoplus_{v \in \sinks(E)} \K(\ell^2(E^* v))\big) \oplus C^*(F).
\]
To prove this, we first define a Cuntz-Krieger $E$-family
$\{q_v : v \in E^0\}$, $\{t_e : e \in E^1\}$ in
$\big(\bigoplus_{v \in \sinks(E)} \K(\ell^2(E^* v))\big) \oplus
C^*(F)$. We will denote the universal Cuntz-Krieger $F$-family
by $\{p^F_v : v \in F^0\}$, $\{s^F_e : e \in F^1\}$, and we
will denote the matrix units in each $\K(\ell^2(E^* v))$ by
$\{\Theta^v_{\alpha,\beta} : \alpha,\beta \in E^* v\}$. As a
notational convenience, for $v \in E^0 \setminus F^0$, we write
$p^F_v = 0$, and similarly for $e \in E^1 \setminus F^1$, we
write $s^F_e = 0$. For $v \in E^0$, let
\[\textstyle
q_v := \Big(\bigoplus_{w \in \sinks(E)} \sum_{\alpha \in v E^* w} \Theta^w_{\alpha,\alpha}\Big) \oplus p^F_v
\]
and for $e \in E^1$, let
\[\textstyle
t_e := \Big(\bigoplus_{w \in \sinks(E)} \sum_{\alpha \in r(e) E^* w} \Theta^v_{e\alpha,\alpha}\Big) \oplus s^F_e.
\]
Routine calculations show that $\{q_v : v \in E^0\}$, $\{t_e :
e \in E^1\}$ is a Cuntz-Krieger $E$-family. This family clearly
generates $\big(\bigoplus_{v \in \sinks(E)} \K(\ell^2(E^*
v))\big) \oplus C^*(F)$, and each $q_v$ is nonzero because if
$p^F_v = 0$ then $v$ must connect to a sink $w$ in which case
$q_v$ dominates some $\Theta^w_{\alpha,\alpha}$. An application
of the Gauge-Invariant Uniqueness Theorem
\cite[Theorem~2.1]{BPRS} implies that there is an isomorphism
\[\textstyle
\pi_{q,t} \colon C^*(E) \to \Big(\bigoplus_{v \in \sinks(E)} \K(\ell^2(E^*
v))\Big) \oplus C^*(F)
\]
such that $\pi_{q,t}(p_v) = q_v$ and $\pi_{q,t}(s_e) = t_e$.

To complete the proof of
(\ref{it:dir-sums:only-if})~implies~(\ref{it:dir-sums:if}), let
$X \subset \sinks(E)$ denote the subset $\{v \in \sinks(E) :
|E^* v| < \infty\}$,
and for each $v \in X$ let $n_v := |E^* v|$.
We have $\K(\ell^2(E^*v)) = M_{n_v}(\C)$ for each $v \in X$.
Recall that $F$ is row-finite and has no sinks, so
Theorem~\ref{no-unital-quotient-then-graph-alg} implies that
$C^*(F)$ has no unital quotient. For each $v \in \sinks(E)
\setminus X$, the $C^*$-algebra $\K(\ell^2(E^* v))$ is simple
and nonunital. Thus
\[\textstyle
A' := \Big(\bigoplus_{v \in \sinks(E) \setminus X} \K(\ell^2(E^* v))\Big) \oplus C^*(F')
\]
has no finite-dimensional quotients.
We get
\[\textstyle
A \cong C^*(E)
\cong \Big(\bigoplus_{v \in \sinks(E)} \K(\ell^2(E^*v))\Big) \oplus C^*(F)
\cong \Big(\bigoplus_{v \in X} M_{n_v}(\C)\Big) \oplus A'
\]
as required.

To see that
(\ref{it:dir-sums:if})~implies~(\ref{it:dir-sums:only-if}), let
$A = \big(\bigoplus_{x \in X} M_{n_x}(\C) \big) \oplus A'$ as
in~(\ref{it:dir-sums:if}). By
Theorem~\ref{no-unital-quotient-then-graph-alg}, there is a
row-finite graph $E'$ with no sinks such that $C^*(E') \cong
A'$. For each $x \in X$, let $E_x$ be a copy of the graph
\[\begin{tikzpicture}[xscale=1.5]
   \node[inner sep=1pt] (v1) at (0,0) {$v_1$};
   \node[inner sep=1pt] (v2) at (1,0) {$v_2$};
   \node[inner sep=3pt] (dots) at (2,0) {$\cdots$};
   \node[inner sep=1pt] (vn) at (3,0) {$v_{n_x}$};
   \draw[-latex] (v1)--(v2);%
   \draw[-latex] (v2)--(dots);%
   \draw[-latex] (dots)--(vn);%
\end{tikzpicture}\]
A standard argument shows that $C^*(E_x) \cong
M_{n_x}(\C)$. Moreover $E := \big(\bigsqcup_{x \in X} E_x\big)
\sqcup E'$ satisfies
\[\textstyle
C^*(E) \cong \Big(\bigoplus_{x \in X} C^*(E_x)\Big) \oplus C^*(E') \cong A
\]
as required.
\end{proof}

For completeness, we conclude the section with the following
well-known result.

\begin{lemma}\label{lem:fin-dim-graph-alg}
A $C^*$-algebra $A$ is finite dimensional if and only if it is
isomorphic to the $C^*$-algebra of a finite directed graph with
no cycles.
\end{lemma}
\begin{proof}
If $E$ is a finite directed graph with no cycles, then $E^*$ is
finite, and hence $C^*(E) = \cspa\{s_\mu s^*_\nu : \mu,\nu \in
E^*\}$ is finite dimensional.

On the other hand, if $A$ is finite-dimensional, then there
exist an integer $n \ge 1$ and nonnegative integers $d_1,
\dots, d_n$ such that $A \cong \bigoplus^n_{i=1} M_{d_i}(\C)$,
and \cite[Corollary~2.3]{KPR} then implies that $A$ is
isomorphic to the $C^*$-algebra of a finite directed graph with
no cycles.  (Moreover, we remark that the last part of the proof
of Proposition~\ref{prp:direct-sums-as-graph-algs} actually shows that every finite-dimensional
$C^*$-algebra is the $C^*$-algebra of a finite graph with no cycles.)
\end{proof}

\subsection{Obstructions to realizations} \label{necessary-cond-subsec}

Here we present a number of necessary conditions for an AF
algebra to be an ultragraph $C^*$-algebra, an Exel-Laca
algebra, or a graph $C^*$-algebra.
Recall that an ultragraph $C^*$-algebra $C^*(\G)$ is an AF-algebra
if and only if $\G$ has no cycles by \cite[Theorem~4.1]{Tom2}.

\begin{proposition} \label{prop:commutatuveUGA}
Let $\G$ be an ultragraph and suppose that $C^*(\G)$ is an
AF-algebra.  If $C^*(\G)$ is commutative, then the ultragraph
$\G$ has no ultraedges, and $C^*(\G) \cong c_0 (G^0)$.
\end{proposition}

\begin{proof}
It suffices to show that $\G$ has no ultraedges.
Suppose that $e$ is an ultraedge in $\G$, and let $v = s(e)$.
Since $C^*(\G)$ is commutative, we have $p_{r(e)} = s_e^*s_e =
s_es_e^* \leq p_{s(e)}$, and hence $r(e) = \{s(e)\}$.
Thus $e$ is a cycle.
This contradicts the hypothesis that $C^*(\G)$
is an AF-algebra.
\end{proof}

\begin{proposition}\label{prp:ELobstruction}
Let $A$ be an AF-algebra that is also an Exel-Laca algebra.
Then $A$ does not have a quotient isomorphic to $\C$,
and for each $n \in \N$ there is a $C^*$-subalgebra of $A$
isomorphic to $M_n(\C)$.
\end{proposition}

\begin{proof}
There exists an ultragraph $\G = (G^0, \G^1, r, s)$
with bijective $s$ such that $C^*(\G) \cong A$
(see Remark~\ref{rem:ELalgis,,,}).
The ultragraph $\G$ has no cycles.
Let $\{p_v\}_{v\in G^0}$ and $\{s_e\}_{e \in \G^1}$
be the generator of $C^*(\G)$
as in Definition~\ref{dfn:CK-G-fam}.

Suppose, for the sake of contradiction, that
there exists a nonzero \shom $\chi \colon C^*(\G) \to \C$.
Since $\chi$ is nonzero, there exists $v \in G^0$
with $\chi(p_v)\neq 0$. Let $e\in \G^1$ be the unique
ultraedge with $s(e)=v$. Since $\G$ has no cycles, we
have $v \notin r(e)$. Hence $p_v$ is orthogonal to $s_e^*s_e$.
Thus $$| \chi (s_e) |^2 \chi (p_v) = \overline{\chi (s_e)}
\chi(s_e) \chi (p_v) = \chi (s_e^* s_e p_v) = 0,$$ and since
$\chi(p_v)\neq 0$, it follows that $| \chi (s_e) |^2 = 0$ and
$\chi (s_e) = 0$.  But then $\chi (p_v) = \chi (s_es_e^*) =
\chi(s_e) \chi (s_e^*) = 0$, which is a contradiction.  Hence
$C^*(\G)$ has no quotients isomorphic to $\C$.

Let $n\in \N$.  We will construct a $C^*$-subalgebra of
$C^*(\G)$ isomorphic to $M_n(\C)$. Choose $v_1 \in G^0$ and let
$e_1 \in \G^1$ be the unique ultraedge with $s(e_1) = v_1$.
Then choose a vertex $v_2 \in r(e_1)$. Since $\G$ has no cycles,
we have $v_2 \neq v_1$. Continuing in this manner, we
can find distinct vertices $v_1, v_2, \ldots, v_{n} \in G^0$
such that $v_{k+1} \in r(e_k)$ for $k = 1, 2, \ldots , n-1$,
where $e_k \in \G^1$ is the unique ultraedge with $s(e_k) =
v_k$.
For $1 \leq i,j \leq n$, we define
\[
\Theta_{i,j} := s_{e_i} s_{e_{i+1}} \ldots s_{e_{n-1}}
          p_{v_n} s_{e_{n-1}}^* s_{e_{n-2}}^* \ldots s_{e_j}^*.
\]
One can check that $\{ \Theta_{i,j} : 1 \leq i,j \leq n \}$
is a family of matrix units, and thus the $C^*$-subalgebra of $C^*(\G)$
generated by $\{ \Theta_{i,j} : 1 \leq i,j \leq n \}$
is isomorphic to $M_n(\C)$.
\end{proof}

\begin{corollary}
If $A$ is an AF-algebra that is also an Exel-Laca algebra, then
$A$ has a Bratteli diagram $(E, d)$ such that $d_v \geq 2$ for
all $v \in E^0$.
\end{corollary}

\begin{proof}
Since $A$ has no quotient isomorphic to $\C$, the result
follows from Lemma~\ref{fin-quotient-lem-1}.
\end{proof}

\begin{corollary} \label{No-fin-dim-E-L-Cor}
No finite-dimensional $C^*$-algebra is isomorphic to an
Exel-Laca algebra.
\end{corollary}

\begin{definition}
We recall that a $C^*$-algebra $A$ is said to be \emph{Type~I}
if whenever $\pi \colon A \to \Bi (\Hi)$ is a nonzero irreducible
representation, then $\K (\Hi ) \subseteq \pi (A)$.  In the
literature, the terms \emph{postliminary}, \emph{GCR}, and
\emph{smooth} are all synonymous with Type~I.
\end{definition}

\begin{proposition} \label{prp:GAobstruction}
Let $C^*(E)$ be a graph $C^*$-algebra that is also an AF-algebra.
Then every unital quotient of $C^*(E)$ is Type~I
and has finitely many ideals.
\end{proposition}

\begin{proof}
By Lemma~\ref{ideal_quotient},
it suffices to show that if a graph $C^*$-algebra
$C^*(E)$ is a unital AF-algebra
then $C^*(E)$ is Type~I and has finitely many ideals.
Note that $C^*(E)$ is a unital AF-algebra
if and only if $E$ has a finite number of vertices
and no cycles.

We first show that $C^*(E)$ has finitely many ideals.
Since $E$ has no cycles, it satisfies Condition~(K).
Hence any ideal of $C^*(E)$ is of the form $I_{(H,S)}$
for a saturated hereditary subset $H$ of $E^0$
and a subset $S \subseteq E^0$ of the set of
breaking vertices for $H$ \cite[Theorem~3.5]{DT}.
Since the set $E^0$ of vertices of $E$ is finite,
there are only a finite number of such pairs $(H,S)$.
Thus $C^*(E)$ has finitely many ideals.

To prove that $C^*(E)$ is of Type~I, first observe that any
graph with finitely many vertices and no cycles contains a sink
$v$, and the ideal $I_v$ generated by $p_v$ is then a
nontrivial gauge-invariant ideal which is Morita equivalent to
$\C$ and hence of Type~I (see \cite[Proposition~2]{aHRW} and
the subsequent remark in \cite{aHRW}).

We shall show by induction on the number of nonzero ideals of
$C^*(E)$ that $C^*(E)$ is Type~I. Our basis case is when has
just one nontrivial ideal $I$. That is, $C^*(E)$ is simple, and
then the Type~I ideal $I_v$ of the preceding paragraph is
$C^*(E)$ itself, proving the result. Now suppose as an
inductive hypothesis that the result holds whenever $C^*(E)$
has at most $n$ distinct nonzero ideals, and suppose that
$C^*(E)$ has $n+1$ such. Let $v$ be a sink in $E$ and let $I_v$
be the corresponding nonzero Type~I ideal as in the preceding
paragraph. If $C^*(E)/I_v$ is trivial, then $C^*(E) = I_v$ is
of Type~I, so we may assume that $C^*(E)/I_v$ is nonzero. Then
Lemma~\ref{ideal_quotient} implies that $C^*(E)/I_v$ is a
unital AF-algebra that is a graph $C^*$-algebra. Moreover,
$C^*(E)/I_v$ has strictly fewer ideals than $C^*(E)$, so the
inductive hypothesis implies that $C^*(E)/I_v$ is of Type~I.
Since an extension of a Type~I $C^*$-algebra by a Type~I
$C^*$-algebra is Type~I (see \cite[Theorem~5.6.2]{Mur}), it
follows that $C^*(E)$ is of Type~I.
\end{proof}

\begin{theorem} \label{simple-AF-graph-EL}
For a simple AF-algebra $A$ we have the following.
\begin{enumerate}
\item If $A$ is finite dimensional
then $A$ is isomorphic to a graph $C^*$-algebra but not isomorphic to an Exel-Laca algebra.
\item If $A$ is infinite dimensional and unital
then $A$ is isomorphic to an Exel-Laca algebra but not isomorphic to a graph $C^*$-algebra.
\item If $A$ is infinite dimensional and nonunital then $A$
   is isomorphic to a $C^*$-algebra of a row-finite graph with no sinks (which is also isomorphic to the Exel-Laca algebra of a row-finite matrix by Lemma~\ref{row-finite-graphs-matrices-lem}).
\end{enumerate}
In particular, each simple AF-algebra $A$ is isomorphic to
either an Exel-Laca algebra or a graph $C^*$-algebra.
\end{theorem}
\begin{proof}

The statement in (1) follows from Lemma~\ref{lem:fin-dim-graph-alg} and Corollary~\ref{No-fin-dim-E-L-Cor}.

For (2) we observe that if $A$ is simple, infinite dimensional,
and unital, then it follows from
Corollary~\ref{no-nonzero-f-d-quotient-then-EL} that $A$ is
isomorphic to an Exel-Laca algebra. Since $A$ is in particular
unital, to see that $A$ is not a graph $C^*$-algebra, it suffices by
Proposition~\ref{prp:GAobstruction} to show that it is not of
Type~I. If we suppose for contradiction that $A$ is of Type~I,
then as it is simple, we must have $A \cong \K(\Hi)$ for some
Hilbert space $\Hi$. Since $A$ is unital, $\Hi$ and hence
$\K(\Hi)$ must be finite-dimensional, contradicting that $A$ is
infinite dimensional.

The statement in (3) follows from
Theorem~\ref{no-unital-quotient-then-graph-alg}. The final
assertion follows from (1), (2), and (3).
\end{proof}

\begin{corollary} \label{UHF-inf-not-graph}
If $A$ is an infinite-dimensional UHF algebra, then $A$ is not
isomorphic to a graph $C^*$-algebra.
\end{corollary}

\section{A summary of known containments} \label{Venn-sec}

In this section we use our results to describe how various classes of AF-algebras are contained in the classes of graph $C^*$-algebras, Exel-Laca algebras, and ultragraph algebras.  We first examine the simple AF-algebras, where we have a complete description.  Moreover, we see that the simple AF-algebras allow us to distinguish among the four classes of $C^*$-algebras of row-finite graphs with no sinks, graph $C^*$-algebras, Exel-Laca algebras, and ultragraph algebras.  Second, we consider general AF-algebras, and while our description in this case is not complete, we are able to describe how the finite-dimensional and stable AF-algebras are contained in the classes of graph $C^*$-algebras, Exel-Laca algebras, and ultragraph algebras.  Furthermore, we use our results to show that there are numerous other AF-algebras in the various intersections of these classes.

\subsection{Simple AF-algebras}

Consider the following partition of the simple AF-algebras.
\begin{align*}
\textrm{AF}^\text{simple}_{\textit{finite}} &:= \text{finite-dimensional simple AF-algebras} \\
\textrm{AF}^\text{simple}_{\infty, \textit{unital}} &:= \text{infinite-dimensional simple AF-algebras that are unital} \\
\textrm{AF}^\text{simple}_{\infty, \textit{nonunital}} &:= \text{infinite-dimensional simple AF-algebras that are nonunital}
\end{align*}
Theorem~\ref{simple-AF-graph-EL} and
Theorem~\ref{no-unital-quotient-then-graph-alg} imply that
\begin{align*}
\textrm{AF}^\text{simple}_{\infty, \textit{nonunital}} &= \text{simple AF-algebras that are $C^*$-algebras of} \\
& \qquad \quad \text{row-finite graphs with no sinks,} \\
\textrm{AF}^\text{simple}_{\textit{finite}} \cup \textrm{AF}^\text{simple}_{\infty, \textit{nonunital}} &= \text{simple AF-algebras that are graph $C^*$-algebras,} \\
\textrm{AF}^\text{simple}_{\infty, \textit{unital}} \cup \textrm{AF}^\text{simple}_{\infty, \textit{nonunital}}  &= \text{simple AF-algebras that are Exel-Laca algebras} \\
\intertext{\noindent and}
\textrm{AF}^\text{simple}_{\textit{finite}} \cup \textrm{AF}^\text{simple}_{\infty, \textit{unital}} \cup \textrm{AF}^\text{simple}_{\infty, \textit{nonunital}}  &= \text{simple AF-algebras that are ultragraph algebras.}
\end{align*}
Hence these three classes of simple AF-algebras allow us to
distinguish among the four classes of $C^*$-algebras of
row-finite graphs with no sinks, graph $C^*$-algebras,
Exel-Laca algebras, and ultragraph algebras.  However, they do
not allow us to distinguish between the classes of
$C^*$-algebras of row-finite graphs with no sinks and the
intersection of graph $C^*$-algebras and Exel-Laca algebras.
Nor do they allow us to distinguish between the classes of
ultragraph $C^*$-algebras and the union of graph $C^*$-algebras
and Exel-Laca algebras.  To distinguish these classes we will need nonsimple examples.

\medskip

\subsection{More general AF-algebras}

For nonsimple AF-algebras, we cannot give such an explicit description. Nevertheless, in Figure~\ref{fig:Venn}
we present a Venn diagram summarizing the relationships we have
established for finite-dimensional and stable AF-algebras, and also give various examples in the intersections of our classes of graph $C^*$-algebras, Exel-Laca algebras, and
ultragraph $C^*$-algebras.

\begin{figure}[htp!]
\[\begin{tikzpicture}[>=latex]]
  \draw[style=thick,fill=white,opacity=1] (2,0) circle (4);
  \node[anchor=north east] at (5.25,2.8) {\begin{tabular}{c}\textsc{Exel-Laca}\\ \textsc{AF-algebras}\end{tabular}};
  \draw[style=semithick] (-8,-6) rectangle (8,6);
  \node[anchor=north west] at (-7.3,5.8) {\textsc{AF-algebras}};
  \draw[style=thick] (-2,0) circle (4);
  \node[anchor=north west] at (-5.25,2.8) {\begin{tabular}{c}\textsc{Graph}\\ \textsc{AF-algebras}\end{tabular}};
  \draw[style=thick, smooth cycle] plot[tension=0.7] coordinates{(-5.75,-4) (-5.75,4.1) (5.75,4.1) (5.75,-4)};
  \node at (0,4.75) {\textsc{Ultragraph AF-algebras}};
  \draw[style=thick, smooth cycle] plot[tension=0.7] coordinates{(-1,-1.75) (-1,-0.25) (1,-0.25) (1,-1.75)};
  \node[anchor=north west, inner sep=3pt] at (-1,-0.25) {\begin{tabular}{c} \text{stable} \\ \text{AF-algs} \end{tabular}};
  \draw[style=thick, smooth cycle] plot[tension=0.7] coordinates{(-1.4,-1.5) (-1.4,1.5) (0,2.5) (1.4,1.5) (1.4,-1.5) (0,-2.5)};
  \node at (0,1.25) {\begin{tabular}{c} \text{AF-algs} \\ \text{of row-finite} \\ \text{graphs with} \\ \text{no sinks} \end{tabular}};
  \draw[style=thick, smooth cycle] plot[tension=0.7] coordinates{(-5,-1.5) (-5,0.5) (-3,0.5) (-3,-1.5)};
  \node[anchor=north west, inner sep=1pt] at (-5.25,0.5) {\begin{tabular}{c} \text{finite} \\ \text{dimensional} \\ \text{$C^*$-algs} \end{tabular}};
  \draw (-2.8,5.5) circle (0.5em) node {\small a};
  \draw (-2.8,3.25) circle (0.5em) node {\small b};
  \draw (0,4.15) circle (0.5em) node {\small c};
  \draw (0,2.85) circle (0.5em) node {\small d};
  \draw (1.15,0.2) circle (0.5em) node {\small e};
  \draw (2.2,3.25) circle (0.5em) node {\small f};
\end{tikzpicture}\]
\caption{A Venn diagram summarizing AF-algebra containments}\label{fig:Venn}
\end{figure}
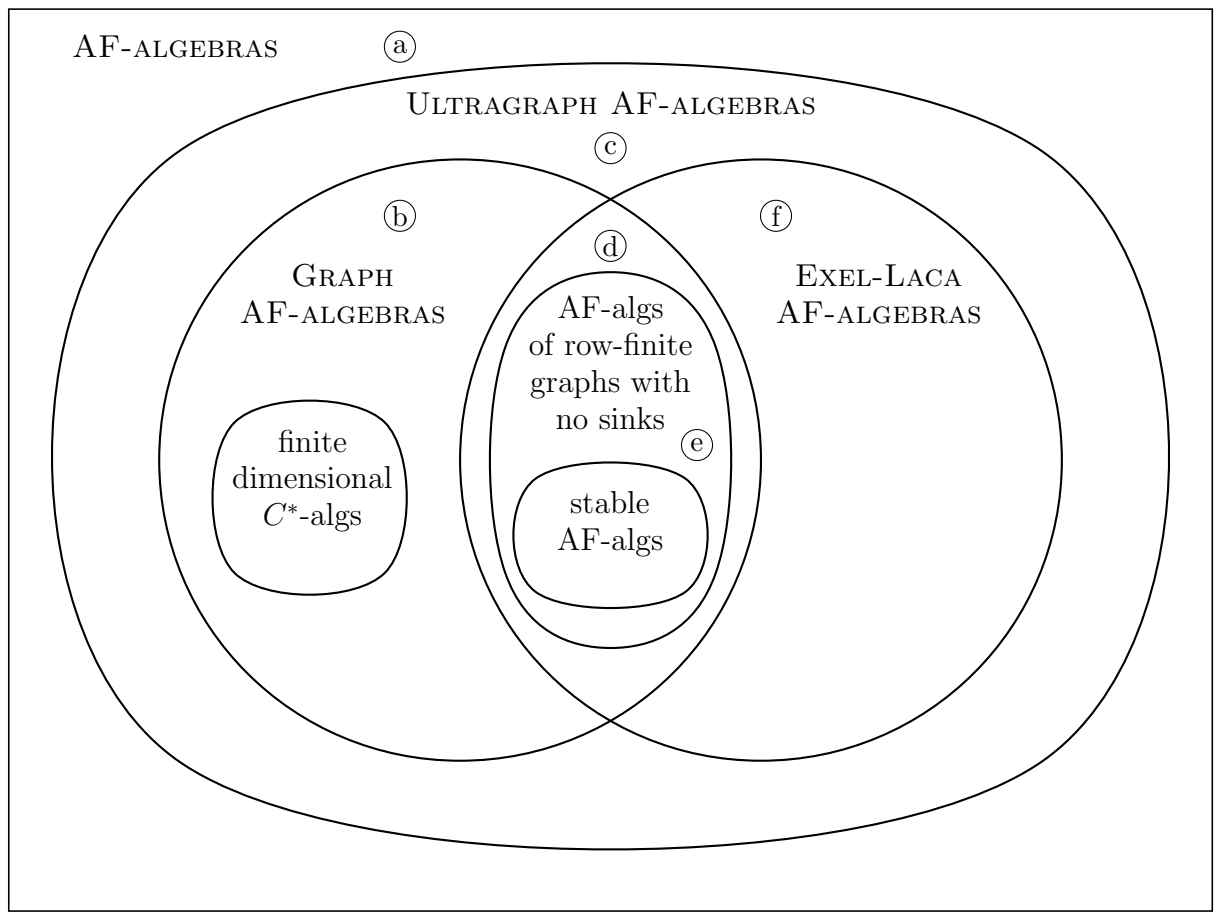

\begin{table}[htp!]
\[\begin{tabular}{||c|c|c||}\hline
\raisebox{0pt}[1em][0em]{Region} & unital $C^*$-algebra& nonunital $C^*$-algebra\\[0.25ex]
\hline
\raisebox{0pt}[1em][0em]{(a)} & $c_c$ & $c_0 \oplus c_c$ \\[0.25ex]
\raisebox{0pt}[1em][0em]{(b)} & $\K^+$ & $c_0$ \\[0.25ex]
\raisebox{0pt}[1em][0em]{(c)} & $M_{2^\infty} \oplus \C$ & $M_{2^\infty} \oplus \C \oplus \K$ \\[0.25ex]
\raisebox{0pt}[1em][0em]{(d)} & $M_2(\MU{\K})$ & $M_2(\MU{\K}) \oplus \K$ \\[0.25ex]
\raisebox{0pt}[1em][0em]{(e)} & --- & $C^*(F_2)$ \\[0.25ex]
\raisebox{0pt}[1em][0em]{(f)} & $M_{2^\infty}$ & $M_{2^\infty} \oplus \K$ \\[0.25ex]\hline
\end{tabular}\]
$ $ \caption{Examples of $C^*$-algebras lying in each region of
Figure~\ref{fig:Venn}}\label{tab:egs-for-diagram}
\end{table}

Table~\ref{tab:egs-for-diagram} presents, for each region of
the Venn diagram of Figure~\ref{fig:Venn}, both a unital and a
nonunital example belonging to that region, with three
exceptions: we give no examples of finite-dimensional or stable
AF algebras, nor any example of a unital AF algebra which is
the $C^*$-algebra of a row-finite graph with no sinks. Our
reasons for these omissions are as follows: examples of
finite-dimensional and stable AF algebras are obvious, and
necessarily unital and nonunital respectively; and no unital
example exists in region (e) by
Theorem~\ref{no-unital-quotient-then-graph-alg}.

In Table~\ref{tab:egs-for-diagram}, we use the following
notation:
\begin{itemize}
\item $M_{2^\infty}$ denotes the UHF algebra of type
   $2^\infty$.
\item $\K$ denotes the compact operators on a separable
   infinite-dimensional Hilbert space.
\item $\MU{\K}$ denotes the minimal unitization of the
   $C^*$-algebra $\K$.
\item $c_0$ denotes the space $\{f : \N \to \C \ | \
   \lim_{n \to \infty} f(n) = 0\}$.
\item $c_c$ denotes the space $\{f : \N \to \C \ | \
   \lim_{n \to \infty} f(n) \in \C \}$.
\item $F_2$ denotes the graph\quad
   \raisebox{-0.5ex}[2ex][0.5ex]{\begin{tikzpicture}[xscale=1.5]
  \node[inner sep=1pt] (a) at (0,0) {\small $v_1$};%
  \node[inner sep=1pt] (b) at (1,0) {\small $v_2$};%
  \node[inner sep=1pt] (c) at (2,0) {\small $v_3$};%
  \node[inner sep=1pt] (d) at (3,0) {\small $v_4$};%
  \node at (3.35,0) {\dots};%
  \draw[-latex] (a) .. controls (0.5,0.15) .. (b);%
  \draw[-latex] (a) .. controls (0.5,-0.15) .. (b);%
  \draw[-latex] (b) .. controls (1.5,0.15) .. (c);%
  \draw[-latex] (b) .. controls (1.5,-0.15) .. (c);%
  \draw[-latex] (c) .. controls (2.5,0.15) .. (d);%
  \draw[-latex] (c) .. controls (2.5,-0.15) .. (d);%
  \end{tikzpicture}}.
\end{itemize}

\noindent We now justify that the examples listed have the
desired properties.

\begin{enumerate}\renewcommand{\theenumi}{\alph{enumi}}
\item
   \begin{itemize}
   \item The unital AF-algebra $c_c$ is not an
       ultragraph $C^*$-algebra since it is
       commutative and its spectrum is not discrete
       (see Proposition~\ref{prop:commutatuveUGA}).
   \item The nonunital AF-algebra $c_0\oplus c_c$ is
       not an ultragraph algebra for precisely the
       same reason that $c_c$ is not.
   \end{itemize}
\item
   \begin{itemize}
   \item The minimal unitization $\MU{\K}$ of the
       compact operators is isomorphic to the
       $C^*$-algebra of the graph\
       \raisebox{-3pt}[2.5ex][1.5ex]{\begin{tikzpicture}[scale=2]
       \node[circle,inner sep=1pt] (v) at (0,0) {\small$v$};%
       \node[circle,inner sep=1pt] (w) at (1,0) {\small$w$};%
       \draw[-latex] (v.20) -- (w.163) node[pos=0.5,anchor=south,inner sep=1.5pt] {\tiny$(\infty)$};%
       \draw[-latex] (v.340) -- (w.197);
       \end{tikzpicture}}\  with two vertices
       $v,w$ and infinitely many edges from $v$ to
       $w$. Since, $\MU{\K}$ has a quotient isomorphic
       to $\C$, it is not an Exel-Laca algebra by
       Proposition~\ref{prp:ELobstruction}.
   \item The nonunital AF-algebra $c_0$ is the
       $C^*$-algebra of the graph with infinitely many
       vertices and no edges. It is not an Exel-Laca
       algebra by Proposition~\ref{prp:ELobstruction}.
   \end{itemize}
\item
   \begin{itemize}
   \item Since $M_{2^\infty}$ is an
       infinite-dimensional simple AF-algebra,
       Theorem~\ref{simple-AF-graph-EL} implies that
       $M_{2^\infty}$ is an Exel-Laca algebra and
       hence also an ultragraph algebra. In addition,
       $\C$ is a graph $C^*$-algebra so also an
       ultragraph $C^*$-algebra. Since the class of
       ultragraph $C^*$-algebras is closed under
       direct sums, $M_{2^\infty} \oplus \C$ is a
       unital ultragraph $C^*$-algebra. It is not an
       Exel-Laca algebra because it has a quotient
       isomorphic to $\C$ (see
       Proposition~\ref{prp:ELobstruction}), and it is
       not a graph $C^*$-algebra because it has a
       unital quotient $M_{2^\infty}$ that is not
       Type~I (see
       Proposition~\ref{prp:GAobstruction}).
   \item Since $\K$ and $M_{2^\infty} \oplus \C$ are
       both ultragraph $C^*$-algebras, the direct sum
       $M_{2^\infty} \oplus \C \oplus \K$ is a
       nonunital ultragraph $C^*$-algebra. It is
       neither a graph $C^*$-algebra nor an Exel-Laca
       algebra as above.
   \end{itemize}
\item
   \begin{itemize}
   \item The unital AF-algebra $M_2(\MU{\K})$ is
       isomorphic to the $C^*$-algebra of the
       following graph
      
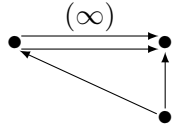
\[\begin{tikzpicture}[scale=2]
          \node[circle,inner sep=-0.5pt] (v) at (0,0) {$\bullet$};%
          \node[circle,inner sep=-0.5pt] (w) at (1,0) {$\bullet$};%
          \draw[-latex] (v.north east) -- (w.north west) node[pos=0.5,anchor=south,inner sep=1.5pt] {\small$(\infty)$};%
          \draw[-latex] (v.south east) -- (w.south west);
          \node[circle,inner sep=-0.5pt] (x) at (1,-0.5) {$\bullet$};%
          \draw[-latex] (x)--(v.-60);
          \draw[-latex] (x)--(w);
      \end{tikzpicture}\]
   and it is also isomorphic to the Exel-Laca algebra of
   the matrix
   \[\left(\begin{tabular}{cccccccc}
       0&1&1&1&1&$\cdots$ \\
       0&0&1&0&0& \\
       0&0&0&1&0& \\
       0&0&0&0&1& \\
       \vdots&&&&&$\ddots$
   \end{tabular}\right).\]
   It is not isomorphic to the $C^*$-algebra of a
   row-finite graph with no sinks by
   Theorem~\ref{no-unital-quotient-then-graph-alg}.
   \item The nonunital AF-algebra $M_2(\MU{\K}) \oplus
       \K$ is isomorphic to both a graph $C^*$-algebra and
       an Exel-Laca algebra because its two direct
       summands have this property. It is not the
       $C^*$-algebra of a row-finite graph with no
       sinks by
       Theorem~\ref{no-unital-quotient-then-graph-alg}
       because it admits the unital quotient
       $M_2(\MU{\K})$.
   \end{itemize}
\item
   \begin{itemize}
   \item There is no unital example in this region by
       Theorem~\ref{no-unital-quotient-then-graph-alg}.
   \item Let $F_2$ denote the graph\quad
   \raisebox{-0.5ex}[2ex][0.5ex]{\begin{tikzpicture}[xscale=1.5]
  \node[inner sep=1pt] (a) at (0,0) {\small $v_1$};%
  \node[inner sep=1pt] (b) at (1,0) {\small $v_2$};%
  \node[inner sep=1pt] (c) at (2,0) {\small $v_3$};%
  \node[inner sep=1pt] (d) at (3,0) {\small $v_4$};%
  \node at (3.35,0) {\dots};%
  \draw[-latex] (a) .. controls (0.5,0.15) .. (b);%
  \draw[-latex] (a) .. controls (0.5,-0.15) .. (b);%
  \draw[-latex] (b) .. controls (1.5,0.15) .. (c);%
  \draw[-latex] (b) .. controls (1.5,-0.15) .. (c);%
  \draw[-latex] (c) .. controls (2.5,0.15) .. (d);%
  \draw[-latex] (c) .. controls (2.5,-0.15) .. (d);%
  \end{tikzpicture}}.  Then $C^*(F_2)$ is a graph $C^*$-algebra, and
       since $F_2$ is cofinal with no cycles and no
       sinks, $C^*(F_2)$ is simple by
       \cite[Corollary~3.10]{KPR}. In addition,
       $C^*(F_2)$ is nonunital because $F_2$ has
       infinitely many vertices. Since $C^*(F_2)$ is
       the $C^*$-algebra of a row-finite graph with no
       sinks, it is both a graph $C^*$-algebra and an
       Exel-Laca algebra (see
       Lemma~\ref{row-finite-graphs-matrices-lem}).
       The function $g : F_2^0 \to \R^+$ defined by
       $g(v_i) = 2^{-i}$ is a graph trace with norm 1
       (see \cite[Definition~2.2]{Tomforde2004}), and
       the existence of such a function implies that
       $C^*(F_2)$ is not stable by
       (a)${}\implies{}$(c) of
       \cite[Theorem~3.2]{Tomforde2004}.
   \end{itemize}
\item
   \begin{itemize}
   \item As in example~(c), the unital AF-algebra
       $M_{2^\infty}$ is an Exel-Laca algebra but not
       a graph $C^*$-algebra.
   \item As in example~(c), the nonunital AF-algebra
       $M_{2^\infty} \oplus \K$ is an Exel-Laca
       algebra but not a graph $C^*$-algebra.
   \end{itemize}
\end{enumerate}


\begin{thebibliography}{00}

\bibitem{BHRS} T. Bates, J.H. Hong, I. Raeburn, and W.
 Szyma\'nski, \emph{The ideal structure of the $C^*$-algebras of
 infinite graphs}, Illinois J. Math. {\bf46} (2002),
 1159--1176.

\bibitem{BPRS} T.~Bates, D.~Pask, I.~Raeburn and
  W.~Szyma\'nski, \emph{The $C^*$-algebras of row-finite
  graphs}, New York J. Math. \textbf{6} (2000), 307--324.

\bibitem{Bra} O.~Bratteli, \emph{Inductive limits of finite
  dimensional $C^*$-algebras}, Trans.~Amer.~Math.~Soc.
  \textbf{171} (1972), 195--234.


\bibitem{CK} J.~Cuntz and W.~Krieger, \emph{A class of
  $C^*$-algebras and topological Markov chains}, Invent.
  Math. \textbf{56} (1980), 251--268.

\bibitem{Dav} K.~Davidson, $C^*$-algebras by Example, Fields
  Institute Monographs, vol. 6, Amer. Math. Soc., Providence,
  1996.

\bibitem{DHS} K.~Deicke, J.H.~Hong and W.~Szyma\'nski,
  \emph{Stable rank of graph algebras. Type~I graph algebras
  and their limits}, Indiana Univ. Math. J. \textbf{52} (2003),
  no. 4, 963--979.

\bibitem{Drinen2000} D. Drinen, \emph{Viewing {AF}-algebras as
  graph algebras}, Proc. Amer. Math. Soc. \textbf{128}
  (2000), 1991--2000.

\bibitem{DT} D.~Drinen and M.~Tomforde, \emph{The
  $C^*$-algebras of arbitrary graphs}, Rocky Mountain
  J.~Math. \textbf{35} (2005), 105--135.

\bibitem{Eff} E.~G. Effros, Dimensions and {$C\sp{\ast}
  $}-algebras, Conference Board of the Mathematical Sciences,
  Washington, D.C., 1981, v+74.

\bibitem{EL} R.~Exel and M.~Laca, \emph{Cuntz-Krieger algebras
  for infinite matrices}, J. reine angew. Math. \textbf{512}
  (1999), 119--172.

\bibitem{EL2} R.~Exel and M.~Laca, \emph{The $K$-theory of
  Cuntz-Krieger algebras for infinite matrices}, $K$-Theory
  \textbf{19} (2000), 251--268.

\bibitem{FLR} N.~Fowler, M.~Laca, and I.~Raeburn, \emph{The
  ${C}^*$-algebras of infinite graphs}, Proc. Amer. Math.
  Soc. \textbf{8} (2000), 2319--2327.

\bibitem{GH} K.~R.~Goodearl and D.~E.~Handelman,
  \emph{Classification of ring and $C^*$-algebra direct limits
  of finite-dimensional semisimple real algebras}.
  Mem.~Amer.~Math.~Soc. \textbf{69} (1987).

\bibitem{aHRW} A.~an~Huef,  I.~Raeburn, and D.~P.~Williams,
  \emph{Properties preserved under Morita equivalence of
  $C^*$-algebras},  Proc. Amer. Math. Soc. \textbf{135}
  (2007), 1495--1503.

\bibitem{KMST} T.~Katsura, P.~Muhly, A.~Sims, and M.~Tomforde,
  \emph{Ultragraph C*-algebras via topological quivers},
  Studia Math, \textbf{187} (2008), 137--155.

\bibitem{KMST2} T.~Katsura, P.~Muhly, A.~Sims, and M.~Tomforde,
  \emph{Graph algebras, Exel-Laca algebras, and ultragraph
  algebras coincide up to Morita equivalence}, J. reine angew.
  Math., to appear.

\bibitem{KPR} A.~Kumjian, D.~Pask, and I.~Raeburn,
  \emph{Cuntz-Krieger algebras of directed graphs}, Pacific
  J. Math. \textbf{184} (1998), 161--174.

\bibitem{KPRR} A.~Kumjian, D.~Pask, I.~Raeburn, and J.~Renault,
  \emph{Graphs, groupoids, and Cuntz-Krieger algebras}, J.
  Funct. Anal. \textbf{144} (1997), 505--541.

\bibitem{Mur} G.~J.~Murphy, $C^*$-algebras and Operator Theory,
  Academic Press, San Diego, 1990.

\bibitem{Rae} I.~Raeburn,  \emph{Graph algebras}. CBMS Regional
  Conference Series in Mathematics, \textbf{103}, Published
  for the Conference Board of the Mathematical Sciences,
  Washington, DC; by the American Mathematical Society,
  Providence, RI, 2005. vi+113 pp.

\bibitem{Szy4} W.~Szyma\'nski, \emph{Simplicity of
  Cuntz-Krieger algebras of infinite matrices}, Pacific J.
  Math. \textbf{199} (2001), 249--256.

\bibitem{Tom} M.~Tomforde, \emph{A unified approach to
 Exel-Laca algebras and $C^*$-algebras associated to
 graphs},  J. Operator Theory {\bf50} (2003), 345--368.

\bibitem{Tom2} M.~Tomforde, \emph{Simplicity of ultragraph
 algebras}, Indiana Univ.~Math.~J. \textbf{52} (2003),
 901--926.

\bibitem{Tomforde2004} M. Tomforde, \emph{Stability of {$C\sp
  \ast$}-algebras associated to graphs}, Proc. Amer. Math.
  Soc. \textbf{132} (2004), no.~6, 1787--1795 (electronic).

\bibitem{Tom9} M.~Tomforde, \emph{Structure of graph
  C*-algebras and their generalizations}, Chapter in the book
  ``Graph  Algebras: Bridging the gap between analysis and
  algebra", Eds.~Gonzalo Aranda Pino, Francesc Perera
  Dom\`enech, and Mercedes Siles Molina, Servicio de
  Publicaciones de la  Universidad de M\'alaga, M\'alaga,
  Spain, 2006.

\end{thebibliography}
\end{document}